\documentclass[11pt]{amsart}              \parindent0pt
\usepackage{graphicx, amssymb,amsmath}
\usepackage[T1]{fontenc}
\usepackage[utf8]{inputenc}
\usepackage{enumitem}
\usepackage{hyperref}
\usepackage{xurl}
\usepackage[labelfont=rm]{caption}
\usepackage[labelformat=simple]{subcaption}
\usepackage[defaultlines=3,all]{nowidow}
\usepackage{tikz}

\usepackage{pgf}
\usepackage{xcolor}

\usetikzlibrary{calc, decorations.markings, intersections, cd, external}
\DeclareMathOperator{\vol}{Vol}
\DeclareMathOperator{\aut}{Aut}
\newcommand{\multiconj}{\mathcal{A}}
\newcommand{\novacio}{\mathcal{F}}
\newcommand{\R}{\mathbb{R}}
\newcommand{\abs}[1]{\left\lvert #1\right\rvert}
\newtheorem{thm}{Theorem}[section]
\newtheorem{proposition}[thm]{Proposition}
\newtheorem{coro}[thm]{Corollary}
\newtheorem{lema}[thm]{Lemma}
\theoremstyle{definition}
\newtheorem{dfn}[thm]{Definition}
\newtheorem{ntc}[thm]{Notation}
\newtheorem{ejm}[thm]{Example}
\theoremstyle{remark}
\newtheorem{remark}[thm]{Remark}

\numberwithin{equation}{section}

\tikzset{->-/.style={decoration={markings, mark=at position #1 with {\arrow{>}}},postaction={decorate}}}

\newcommand\cruz[2]{
\begin{scope}[shift={(#1)}]
\draw (-1,-1) rectangle (1,1);
\draw (-1,-1) --(1,1);
\draw (-1,1) --(1,-1);
#2
\end{scope}
}

\newcommand\cuadradob[3]{
\begin{scope}[shift={(#1,#2)},scale=.9]
\draw (-1,-1) rectangle (1,1);
\draw (-1,-1) node[left] {$A_4$} -- (1,1) node[right] {$A_2$} (1, -1)node[right] {$A_1$} --(-1,1) node[left] {$A_3$};
\fill[red] (#3) circle [radius=.1];
\end{scope}
}

\newcommand\hexagono[1]{
\begin{scope}[shift={(0,0)}]
\foreach \x in {1,...,6}
{
\coordinate (A\x) at ({60*(\x-3)}:1.5);
\fill (A\x) circle [radius=.1];
\node at ($1.26*(A\x)$) {$A_\x$};
}
\foreach \x [evaluate=\x as \y using \x+1] in {1,...,5}
\foreach \z in {\y,...,6}
{
\draw[name path=(L\x-\z)] (A\x) -- (A\z);
}
\fill[red] (#1) circle [radius=.1];
\end{scope}
}

\newcommand\esfera{
\filldraw[fill=cyan!50!white] circle [radius=1cm];
\draw (-1,0) arc [start angle=180,end angle=360,x radius=1cm,y radius=.25cm];
\draw[dashed] (-1,0) arc [start angle=180,end angle=0,x radius=1cm,y radius=.25cm];
\fill[red] (-110:1) circle [radius=.1];
\fill[red] (-70:1) circle [radius=.1];
\fill[red] (-20:1) circle [radius=.1];
\fill[red] (20:1) circle [radius=.1];
}

\newcommand\toro{
\filldraw[fill=cyan!50!white] (0,-.75) to[out=90,in=-90] (-.25,-.5) to[out=90,in=-90] (0,-.25)to[out=90,in=-90] (-.25,0) to[out=90,in=-90] (0,.25)to[out=90,in=-90] (-.25,.5) to[out=90,in=-90] (0,.75) to[out=90,in=90] (-1,0) to[out=-90,in=-90] (0,-.75);
\filldraw[fill=white] (-.625,0) circle [radius=.125cm];
\draw (-1,0) arc [start angle=180,end angle=360,x radius=.125cm,y radius=.05cm];
\draw[dotted] (-1,0) arc [start angle=180,end angle=0,x radius=.125cm,y radius=.05cm];
\draw (-.5,0) arc [start angle=180,end angle=360,x radius=.125cm,y radius=.05cm];
\draw[dotted] (-.5,0) arc [start angle=180,end angle=0,x radius=.125cm,y radius=.05cm];
\fill[red] (0,-.75) circle [radius=.05];
\fill[red] (0,-.25) circle [radius=.05];
\fill[red] (0,.25) circle [radius=.05];
\fill[red] (0,.75) circle [radius=.05];
}

\newcommand\esfd{
\filldraw[fill=cyan!50!white] (.25,0) to[out=180,in=0] (0,.25) to[out=180,in=0] (-.25,0)to[out=180,in=-90] (-.5,.5) to[out=90,in=180] (0,1)to[out=0,in=90] (.5,.5) to[out=-90,in=0] (.25,0);
\draw (-.5,.5) arc [start angle=180,end angle=360,x radius=.5cm,y radius=.1cm];
\draw[dotted] (-.5,.5) arc [start angle=180,end angle=0,x radius=.5cm,y radius=.1cm];
\fill[red] (.25,0) circle [radius=.05];
\fill[red] (-.25,0) circle [radius=.05];
\fill[red] (.475,.25) circle [radius=.05];
\fill[red] (.425,.75) circle [radius=.05];
}

\newcommand\esff{
\filldraw[fill=cyan!50!white] (.25,0) to[out=180,in=0] (0,.25) to[out=180,in=0] (-.25,0)to[out=180,in=-90] (-.5,.5) to[out=90,in=180] (0,1)to[out=0,in=90] (.5,.75) to[out=-90,in=90] (.25,.5)to[out=-90,in=90] (.5,.25) to[out=-90,in=0] (.25,0);
\draw (-.5,.5) arc [start angle=180,end angle=360,x radius=.375cm,y radius=.1cm];
\draw[dotted] (-.5,.5) arc [start angle=180,end angle=0,x radius=.375cm,y radius=.1cm];
\fill[red] (.25,0) circle [radius=.05];
\fill[red] (-.25,0) circle [radius=.05];
\fill[red] (.5,.25) circle [radius=.05];
\fill[red] (.5,.75) circle [radius=.05];
}

\newcommand\bloque[4]{\filldraw[fill=green!80!white, draw=black,xshift=#1 cm,yshift=#2 cm,xscale=#3,yscale=#4]
(-1,0) --(-1/2,0) -- (-1/2,1/4) -- (-3/4,1/2) -- (-1,1/2) --cycle;}

\newcommand\bloqueiv[1]{
\begin{scope}[shift={#1}]
\bloque{0}{0}{1}{1}
\bloque{0}{1}{1}{-1}
\bloque{-1}{0}{-1}{1}
\bloque{-1}{1}{-1}{-1}
\fill[fill=green!20!white]
(-1/2,1/4) -- (-1/4,1/2) -- (-1/2,3/4) -- (-3/4,1/2) --cycle;
\end{scope}}

\pgfmathsetmacro{\r}{2}
\pgfmathsetmacro{\s}{{2/3}}
\newcommand\borromeo{\coordinate (A1) at (-1,0);
\coordinate (A2) at (1,0);\
\coordinate (A3) at (0,{sqrt(3)});

\coordinate (P1) at (0,{sqrt(3)});
\coordinate (P2) at (-1,0);
\coordinate (P3) at (1,0);
\coordinate (P4) at (2,{sqrt(3)});
\coordinate (P5) at (-2,{sqrt(3)});
\coordinate (P6) at (0,{-sqrt(3)});

\foreach \x in {1,2,3}
{
\draw[name path global/.expanded=C\x] (A\x) circle [radius=\r cm];
}

\foreach \x in {1,...,6}
{
\draw[dashed] (P\x) circle [radius=\s cm];
\path[name path global/.expanded=D\x] (P\x) circle [radius=\s cm];
}
}
\newcommand{\borromeoa}{\coordinate (A1) at (-1,0);
\coordinate (A2) at (1,0);\
\coordinate (A3) at (0,{sqrt(3)});

\coordinate (P1) at (0,{sqrt(3)});
\coordinate (P2) at (-1,0);
\coordinate (P3) at (1,0);
\coordinate (P5) at (2,{sqrt(3)});
\coordinate (P4) at (-2,{sqrt(3)});
\coordinate (P6) at (0,{-sqrt(3)});

\foreach \x in {1,2,3}
{
\draw[name path global/.expanded=C\x] (A\x) circle [radius=\r cm];
}

\foreach \x in {1,...,6}
{
\fill[white] (P\x) circle [radius=\s cm];
\path[name path global/.expanded=D\x] (P\x) circle [radius=\s cm];
}

\path[name intersections={of=C1 and D1, name=Q1, total=\ti}];
\coordinate (R1-4) at (Q1-1);
\coordinate (R1-3) at (Q1-2);
\path[name intersections={of=C2 and D1, name=Q2, total=\ti}];
\coordinate (R1-5) at (Q2-1);
\coordinate (R1-2) at (Q2-2);

\path[name intersections={of=C2 and D2, name=Q3, total=\tii}];
\coordinate (R2-1) at (Q3-1);
\coordinate (R2-6) at (Q3-2);
\path[name intersections={of=C3 and D2, name=Q4, total=\tii}];
\coordinate (R2-4) at (Q4-1);
\coordinate (R2-3) at (Q4-2);

\path[name intersections={of=C1 and D3, name=Q5, total=\tii}];
\coordinate (R3-1) at (Q5-1);
\coordinate (R3-6) at (Q5-2);
\path[name intersections={of=C3 and D3, name=Q6, total=\tii}];
\coordinate (R3-2) at (Q6-2);
\coordinate (R3-5) at (Q6-1);

\path[name intersections={of=C1 and D4, name=Q7, total=\tii}];
\coordinate (R4-1) at (Q7-1);
\coordinate (R4-6) at (Q7-2);
\path[name intersections={of=C3 and D4, name=Q8, total=\tii}];
\coordinate (R4-2) at (Q8-2);
\coordinate (R4-5) at (Q8-1);

\path[name intersections={of=C2 and D5, name=Q9, total=\tii}];
\coordinate (R5-1) at (Q9-1);
\coordinate (R5-6) at (Q9-2);
\path[name intersections={of=C3 and D5, name=Q10, total=\tii}];
\coordinate (R5-3) at (Q10-2);
\coordinate (R5-4) at (Q10-1);

\path[name intersections={of=C1 and D6, name=Q11, total=\tii}];
\coordinate (R6-3) at (Q11-1);
\coordinate (R6-4) at (Q11-2);
\path[name intersections={of=C2 and D6, name=Q12, total=\tii}];
\coordinate (R6-5) at (Q12-2);
\coordinate (R6-2) at (Q12-1);
}

\newcommand\dsa[4]{
\draw #1--#2;
\draw #3-- #4;
\draw[line width=1.2,red] ($.5*#1+.5*#2$) -- 
($.5*#3+.5*#4$);
\fill ($.5*#1+.5*#2$) circle [radius=.1] ;
\fill ($.5*#3+.5*#4$)circle [radius=.1] ;
}
 
\title{On generic singularities of intersections of ellipsoids: the octahedron}    \author[E. Artal]{Enrique Artal Bartolo}
\address[E.~Artal]{Departamento de Matem\'{a}ticas, IUMA \\
Universidad de Zaragoza \\
C.~Pedro Cerbuna 12, 50009, Zaragoza, Spain}
\urladdr{\url{http://riemann.unizar.es/~artal}}
\email{\href{mailto:artal@unizar.es}{artal@unizar.es}}

\author[S. L\'{o}pez de Medrano]{Santiago L\'{o}pez de Medrano}
\address[S. L\'{o}pez de Medrano]{Instituto de Matem\'{a}ticas\\
Universidad Nacional Aut\'{o}noma de M\'{e}xico\\
04510 Ciudad de M\'{e}xico, M\'{e}xico}
\email{\href{mailto:santiago@matem.unam.mx}{santiago@im.unam.mx}}

\author[M.T. Lozano]{Mar\'{\i}a Teresa Lozano}
\address[M.T. Lozano]{Departamento de Matem\'{a}ticas, IUMA \\
Universidad de Zaragoza \\
C.~Pedro Cerbuna 12, 50009, Zaragoza, Spain}
\email{\href{mailto:tlozano@unizar.es}{tlozano@unizar.es}}

\thanks{\noindent Partially supported by MCIN/AEI/10.13039/501100011033 (grant code: PID2020-114750GB-C31)
and by Departamento de Ciencia, Universidad y Sociedad del Conocimiento del Gobierno de Arag{\'o}n
(grant code: E22\_20R: ``{\'A}lgebra y Geometr{\'i}a''). The two first named authors were also partially supported by UNAM-
Papiit grants IN102918 and IN106324.}
\date{\today}

\begin{document}

\begin{abstract}
The goal of this work is to study the smoothings of singular coaxial
intersections of ellipsoids (where coaxial includes concentric) with
generic singularities, with special attention to the 3-dimensional
case.
\end{abstract}

\maketitle

\section*{Introduction}

In a series of papers, the second named author has extensively studied the intersections of coaxial ellipsoids in $\R^n$ (we will refer to them simply
as \emph{intersections} to simplify the text), specially the smooth ones; see~\cite{LdM2023} and the references therein.
In \cite{LdM2014,LdM-Pepe}, he also
studied some singular ones and in this work we continue their study.
We
are mostly interested in the case of isolated singularities and, more
precisely, in the mildest ones (called \emph{generic}), which are the
analogue to ordinary double points in complex singularity theory. In future works
we will deal with more general singularities.

We will recall how one can express the existence and type of a singularity in terms of the tuples of points defining the equations of the intersection, see~\eqref{eq:inter-1}.
Also, we will recall the definition of the associated polytope of an intersection which in the case of smooth intersections is \emph{simple} \footnote{A vertex of a polytope of dimension $d$ is simple if it is contained in exactly $d$ facets and the polytope is simple is all its vertices are simple}.

The main goal of this work is to study the deformations of a singular
intersection, more precisely, the deformations of its equations giving
nearby smooth intersections and less degenerate singular ones. The
\emph{smoothings} correspond to the connected components of the space of
a deformation consisting only of smooth intersections.

After developing the basic facts about the theory in all dimensions we turn our attention to dimension 3 as we did in
\cite{ALdML:16} for the smooth case. The vertices of a polytope associated to a smooth $3$-intersections are simple and each one lies in exactly three coordinate hyperplanes. The generic singularities of $3$-dimensional intersections come from vertices of the polytope that lie in four coordinate hyperplanes. There are two types of them,  those for which the vertex is simple and those where the vertex lies in~$4$ faces of the polytope, called a $4$-vertex for now. The singularities of the intersections coming from simple or $4$-vertices are  cones over $\mathbb{S}^0\times\mathbb{S}^2$ or $\mathbb{S}^1\times\mathbb{S}^1$, respectively.

Since we will be interested mainly in the simplest cases associated to a
given polyhedron, we will study in detail only examples of the second
type. We will begin by considering the associated intersection of the
simplest one, the quadrangular pyramid. We give a detailed study of the
topology of the singular intersection, its smooth part, and its smoothings.

Two other polyhedra are extensively studied in the work: the
triangular bipyramid and the octahedron. As for the quadrangular pyramid,
we study the topology of their singular intersections, their smooth parts, and their smoothings.
The intersection of the triangular bypiramid has~$12$ singular points and its
homology can be computed using \cite[Theorem~5.21]{LdM2023}:
$H_1$ is trivial and $H_2\cong\mathbb{Z}^{12}$; actually the intersection
is simply connected.

The two possible smoothings are completely described. We pay special attention
to the smooth part of the intersection. It can be described as a Galois cover of the complement
of a simple link in $(\mathbb{S}^1)^3$, and it admits a complete hyperbolic metric, which turns out to
be a topological invariant. This complete hyperbolic manifold has~$12$ ends and 
the singular intersection is a compactification obtained by adding one point to each end.

Since the smooth part of the intersection associated to the
quadrangular bipyramid is homeomorphic to $(\mathbb{S}^1)^2\times\R$, a similar property holds:
it admits a complete euclidean metric and the singular intersection, the suspension
over the torus, is obtained by compactifying the two ends.

The study of the intersection associated to the octahedron is more involved.
The singular intersection has~$96$ singular points. Some data about is topology can be found
in Proposition~\ref{prop:groups-octa}, namely, its fundamental group is $\mathbb{Z}^4$.
As for the triangular bipyramid the smooth part of the intersection
admits also a complete hyperbolic structure with $96$~ends: as before the singular
intersection is obtained by compactifying these ends. Finally, we describe the topology
of the smoothings, obtaining five $3$-manifolds. Three of them are connected sums
of products of spheres; one of these connected sums is obtained through two
combinatorially distinct simple polytopes. A fourth smoothing is the product
of $\mathbb{S}^1$ and a surface of genus~$17$. The last one is a graph manifold
in the sense of Waldhausen. There are two families of distinct links with $96$~components
such that the contraction of each component gives the singular intersection.

An intersection and its polytope are related by a natural orbifold structure
in the polytope, where all the faces are mirrors. The orbifold structure
of the polytope minus the $4$-vertices makes the connection with the geometric
structures via the orbifold structures induced by suitable tessellations (hyperbolic
for the triangular bipyramid and the octahedron and euclidean for the quadrangular
pyramid). These ideas can be extended to other polytopes, e.g., the rhombic dodecahedron.
The three hyperbolic structures on the
triangular bipyramid, the octahedron, and the rhombic dodecahedron,
have a common ancestor in a special tetrahedron, which allows us to compare
the volume of the corresponding hyperbolic manifolds, an important
topological invariant.

One last word about homeomorphism and diffeomorphism. In the smooth case a deformation of the intersection that does not include singular ones preserves its differentiable type. The same seems to be the case when there are only generic singularities and the deformation preserves them as such. But in the general case a deformation that preserves the singularities and their types may preserve only the topological type of the intersection.

In \S\ref{sec:settings} and \S\ref{sec:sing-settings} we introduce the main notions in the paper, including illustrating examples.
Deformations
and smoothings are introduced in \S\ref{sec:smoothings}.
In \S\ref{sec:3smoothings} we present a detailed description of the smoothings
of the intersections associated with the quadrangular pyramid and the triangular bipyramid, including the surgeries related them.
The section~\S\ref{sec:octahedron} is devoted to the study of the smoothings of the intersection associated with the octahedron,
listing the possible ones, their relations, and the geometric properties. In the last section \S\ref{sec:hyperbolic}
we continue the exploration of the relationship of the intersections and some complete hyperbolic manifolds and orbifolds.
 
\section{Settings}
\label{sec:settings}

Let us consider the intersection $Z$ of the points $(x_1,\dots,x_n)\in\R^n$
satisfying
\begin{equation}\label{eq:inter-1}
\sum_{i=1}^{n} A_i x_i^2=0,\qquad
\sum_{i=1}^{n} x_i^2=1,\qquad \text{with}\quad  A_i \in  \R^{m},
\end{equation}
where the first~$m$ equations define a cone and the last one a sphere.
To see that such an intersection coincides with an intersection of coaxial (and concentric) ellipsoids, up to a diagonal automorphism of $\R^n$ and permutation of variables, it is enough to add large enough multiples of the last equation to the $m$ homogeneous ones. Reciprocally, any coaxial collection of $m+1$ ellipsoids with common center at the origin can be put in the above form.
The collection of ellipsoids is determined by the
$n$-tuple $\multiconj:=(A_1,\dots,A_n)\subset(\R^m)^n$.

\begin{ntc}
Up to permutation,
we can represent the tuple $\multiconj$ as $(v_1^{m_1},\dots,v_r^{m_r})$, where $v_1,\dots,v_r\in\R^m$ and the superindices indicate the repetitions of these vectors, i.e., $m_1+\dots + m_r=n$.
\end{ntc}

Let $Z_{\geq 0}$ be the intersection of $Z$ with the first orthant $\R^n_{\geq 0}$ of $\R^n$.
The \emph{variety} $Z$ can be reconstructed from $Z_{\geq 0}$ via
the reflections through the coordinate hyperplanes. The intersection $Z_{\geq 0}$ is homeomorphic to the polytope
\begin{equation}\label{eq:poly}
P:=\left\{(r_1,\dots,r_n)\in\R^n_{\geq 0}
\vphantom{\sum_{i=1}^n}\right.\left| \sum_{i=1}^n A_i r_i=0, \sum_{i=1}^n r_i=1\right\}.
\end{equation}
The union of the reflections  of $P$ through the coordinate hyperplanes is also homeomorphic to $Z$.

\begin{dfn}\label{dgenesing}
An intersection of ellipsoids determined by an $n$-tuple
$\multiconj=(A_1,\dots,A_n)$ in $\R^m$ satisfies the property of
\emph{weak hyperbolicity} or (WH) \footnote{The concept of weak hyperbolicity (WH) (and its equivalence with the transversality of the intersection) was introduced and named by Marc Chaperon~\cite{chaperon} and independently discovered (but not named) in \cite{wall} and \cite{LdM1989} for the case $m=2$.}
if the origin is not a convex combination of a subset of $\leq m$ vectors of $\multiconj$. 
\end{dfn}

This condition is equivalent to the transversality of the equations
and therefore implies that $Z$ is smooth, $P$ is a simple
polytope and that a small deformation of the $n$-tuple will give a
diffeomorphic intersection and a combinatorially equivalent polytope.

When (WH) holds, the faces of $P$ correspond to its non-empty
intersections with the coordinate subspaces and will have the expected
dimension given by transversality. In particular, $\dim Z= \dim P =n-
m-1$. For a deeper understanding of the general case we introduce
now another way to relate $\multiconj$ and $P$ that includes explicitly the
degree of singularity of each face of the polytope.

\begin{dfn}
Let $v_1,\dots,v_r\in\R^m$. A convex combination
$t_1 v_1+\dots+t_r v_r$, $t_i\in\R_{\geq 0}$, $t_1+\dots+t_r=1$,
is \emph{proper} if all the coefficients are positive.
\end{dfn}

Let us denote $[n]:=\{1,\dots,n\}$. For each $J\subset[n]$, we define
\[
C_J:=\left\{(r_1,\dots,r_n)\in P
\right|\left.
r_j>0\text{ if }j\in J, r_j=0\text{ if }j\notin J
\right\}.
\]
We denote by $\novacio_\multiconj$ the set of subsets of $[n]$ such that
$C_J\neq\emptyset$. For $J\in \novacio_\multiconj$, let us denote the closure
$F_J:=\overline{C_J}$.

\begin{lema}\label{lema:vec_poly0}
The faces of the polytope $P$ are $\{F_J\mid J\in \novacio_\multiconj\}$.
Moreover, the dimension of the space of convex combinations of the tuple
$\multiconj_J:=(A_i\mid i\in J)$ is the dimension of the associated face~$F_J$.
\end{lema}

\begin{remark}
(WH) is equivalent
to the following condition:
$\forall J\in \novacio_\multiconj$
\(
\dim F_J+m+1=\#S.
\)
In particular, the vertices of $P$ are associated to sets $J$ of cardinality $m+1$ for which the origin is a proper convex combination of $\multiconj_J$ in a unique way, i.e., $\multiconj_J$ is affinely independent.
\end{remark}

\begin{remark}
A tuple can be deformed by multiplying each $A_i$  by a positive real number~$c_i$. The differential topology of $Z$ is preserved: starting with the equations of $Z$ a diagonal linear change of coordinates in $R^{n}$ will give the new homogeneous equations while changing the equation of the sphere into that of an ellipsoid. Then the intersection of the new cone with this ellipsoid is radially diffeomorphic to its intersection with the unit sphere. So we can assume, for example, that all non-zero $A_i$ belong to the unit sphere in~$\R^m$.
\end{remark}

\begin{ejm}\label{ejm:m=1-liso}
Let us consider $m=1$ and $\multiconj$ satisfying (WH).
Then we can deform $\multiconj$ without changing the topology
in order to have $\multiconj:=((-1)^p,(1)^q)$.
Simple
algebraic operations show that $Z$ is diffeomorphic
to $\mathbb{S}^{p-1}\times\mathbb{S}^{q-1}$ and $P$ is
$\Delta_{p-1}\times\Delta_{q-1}$.
\end{ejm}

\begin{ejm}\label{ejm:m=2-liso}
The case $m=2$ and $\multiconj$ satisfying (WH) can be also easily
described. After deformation,
$\multiconj=(P_1^{r_1},\dots,P_{2k+1}^{r_{2k+1}})$, $k>0$ and
$n=r_1+\dots+r_{2k+1}$,
where $(P_1,\dots,P_{2k+1})$ are equidistributed in the unit circle~\cite{LdM1989}.

If $k=1$, $Z$ is $\mathbb{S}^{r_1-1}\times\mathbb{S}^{r_2-1}\times\mathbb{S}^{r_3-1}$.
If $k>1$ it is a connected sum of products of spheres. For example, for $k=2$ and $r_i=1$, $Z$
is the compact oriented surface of genus~$5$.
\end{ejm}

\begin{ejm}\label{ejm:m=2-liso-total}
There is also a complete classification of smooth connected intersections in dimension~$2$. We obtain
the surfaces of genus~$g_n=2^{n-3}(n-4)+1$ for $n\geq 3$~\cite{LdM2023}.
The fact that these genera appear as intersection of quadrics was communicated
to the second named author by F.~Hirzebruch.
\end{ejm}

We are going now to set three conditions on the intersections which
reduce the number of cases that need to be studied. We state them in
two versions which we check in each case that they are equivalent.

\begin{enumerate}[label=\rm(N\arabic{enumi})]
\item\label{N1}
There are no redundant equations. In terms of $\multiconj$, it generates $\R^m$ as a vector space.

\begin{proof} The equivalence of the two statements in \ref{N1} is because
if the second statement of~\ref{N1} does not hold, let $H\subsetneq\R^m$ be the subspace generated by $\multiconj$, $m':=\dim H<m$. After a linear automorphism of~$\R^m$, we may assume that the last~$m-m'$ coordinates of these vectors vanish. Let $\multiconj'\subset\R^{m'}$ be the tuple obtained by forgetting these last coordinates. The intersections associated to $\multiconj$ and $\multiconj'$ coincide and the only difference is the number of equations.
\end{proof}
From now on, we will assume
\ref{N1} holds.

\item\label{N2} The intersection $Z$ is not contained in a coordinate hyperplane $x_i=0$. In terms of tuples, $\multiconj$ is not contained in a closed half-space of $\R^m$
whose boundary is a linear hyperplane.

\begin{proof}The equivalence of the two statements in \ref{N2} is because
if the second statement of  \ref{N2} does not hold, after a linear automorphism of $\R^m$ we may assume that the $m^{\text{th}}$-coordinates of the points in $\multiconj$ are $\geq 0$.
As \ref{N1} holds, at least one  $m^{\text{th}}$-coordinate must be positive, say for $A_i$;
then $Z\subset\{x_i=0\}$, and we can see this intersection in~$\R^{n-1}$.
\end{proof}
From now on, we will assume
\ref{N2} holds.

It is obvious that if $\multiconj$ is contained in an open half-space whose boundary is a linear hyperplane of $\R^m$,
then the intersection is empty.

\item\label{N3} The intersection $Z$ cuts all the coordinate hyperplanes.
For all $i=1,\dots,n$,
the origin is a convex combination of $(A_1,\dots,\widehat{A_i},\dots,A_n)$.

\begin{proof} The equivalence of the two statements in \ref{N3} is because
if the second statement of  \ref{N3} does not hold, let us assume that whenever the origin can be expressed as a convex combination of $\multiconj$,
the coefficient of $A_j$ does not vanish. With these conditions $Z\cap\{x_j=0\}=\emptyset$.
The projection $Z_j$ of $Z$ in $\{x_j=0\}$ is also an
intersection and $Z\cong Z_j\times\mathbb{S}^0$.
\end{proof}

From now on, we will assume
\ref{N3} holds, unless otherwise stated.
\end{enumerate}

Sometimes we will start from an $d$-dimensional polytope~$P\subset\R^d$
defined by a family $\mathcal{I}$ of $n$~affine inequalities
$p_i(t_1,\dots,t_d)\geq 0$, $i=1,\dots,n$, which define the polytope~$P$.
The image of the affine map $(t_1,\dots,t_d)\mapsto(p_1,\dots,p_n)$
intersects the first orthant of $\R^m$ in the image of $P$
and satisfies affine equations that can be put in the standard form~\eqref{eq:poly}.

When $n$ is the number of facets of $P$ then $P$ gets embedded in $\R^n$
with a facet in each hyperplane we get the best embedding of~$P$, usually known as its \emph{geometric embedding},
that reflects faithfully
(and is essentially determined by) the combinatorics of $P$.

\section{Singular intersections: settings}
\label{sec:sing-settings}

If (WH) does not hold the intersection is not transverse and $Z$ will be singular as an algebraic variety.
The easiest case where (WH) does not hold is when $\mathbf{0}\in\multiconj$ and $m>0$. There is a nice description of the topology of $Z$ and the combinatorics
of $P$ in this case in terms of the one given by $\multiconj\setminus\{\mathbf{0}\}$:

\begin{proposition}[{\cite[p. 162]{LdM-Pepe}, \cite[Example 5.7]{LdM2023}}]
\label{prop:cono}
Let $\multiconj_1$ be a tuple in $\R^{m}$ and let
$\multiconj_2=\multiconj_1\cup(\mathbf{0})$. Let $P_i,Z_i$ be the polytopes
and intersections associated to each tuple. Then,
$P_2$ is the cone over $P_1$ and $Z_2$ is the suspension of $Z_1$.
\end{proposition}

Hence given an intersection $Z$, its suspension is singular (except if $Z=\mathbb{S}^{n-1}$).
If $Z$ is smooth, its suspension has two singular points, each with a neighborhood which is a cone on $Z$.

For $m=0$, there is no singular intersection. For $m=1$, they can be classified since the operation of adding $\mathbf{0}$ is the only one that may induce a singularity.

\begin{ejm}\label{ejm:m=1}
Let
$m=1$ and $\multiconj=(-1,1,0)$. We have three subsets of $S$
with proper convex combinations:
$S_0:=\{0\}$, yielding a vertex $F_3:=(0,0,1)$;
$S_1:=\{1,-1\}$ yielding a vertex $F_{12}=\left(\frac{1}{2},\frac{1}{2},0\right)$;
$S_2:=\multiconj$ yielding the edge $F_{123}$ which coincides with the whole polytope $P$:
\begin{equation*}
P:\
-r_1+r_2+0 =0,\qquad
  r_1+r_2+r_3 =1.
\end{equation*}
Note that this is a non-geometric embedding of $[0,1]\subset\R$. Looking at the numerical invariants there is only a unique singularity at $F_3$.
The intersection $Z$ is given by
\begin{equation*}
Z:\
  -x_1^2+x_2^2+0 =0,\qquad
  x_1^2+x_2^2+x_3^2 =1,
\end{equation*}
which is the union of two orthogonal
meridians of the sphere~$\mathbb{S}^2$ (so it is the suspension of its four points of intersection with
the equator), its
two singular points are the poles and is homeomorphic to the complete bipartite graph $K_{2,4}$ (Figure~\ref{fk24}).

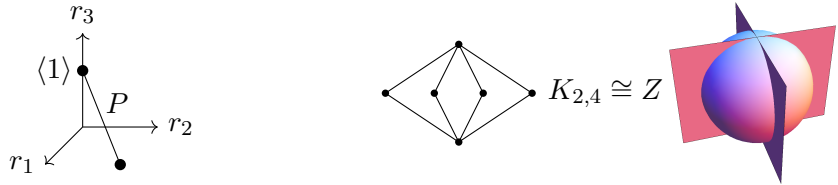
\begin{figure}[ht]
\begin{center}
﻿\begin{tikzpicture}
\coordinate (O) at (0,0);
\coordinate (X) at (-135:{sqrt(2)/2});
\coordinate (Y) at (0:1);
\coordinate (Z) at (90:1.25);
\coordinate (T) at (.5,-.5);
\coordinate (P) at (0,.75);

\draw[->] (O) -- (X) node[left] {$r_1$};
\draw[->] (O) -- (Y) node[right] {$r_2$};
\draw[->] (O) -- (Z) node[above] {$r_3$};
\draw (P) node[left] {$\langle1\rangle$}-- node[right,pos=.35] {$P$} (T);
\fill (T) circle [radius=.075];
\fill (P) circle [radius=.075];

\begin{scope}[xshift=5cm, yshift=0.45cm, scale=.65, rotate=90]
\foreach \a in {1,2}
{
\coordinate (A\a) at ({2*\a-3},0);
\fill (A\a) circle [radius=.075];
}
\foreach \b in {1,...,4}
{
\coordinate (B\b) at (0,{\b-2.5});
\fill (B\b) circle [radius=.075];
}
\foreach \a in {1,2}
\foreach \b in {1,...,4}
{
\draw (A\a) -- (B\b);
}

\node at (0,-3) {$K_{2,4}\cong Z$};

\node at (0,-6) {\includegraphics[height=2.5cm]{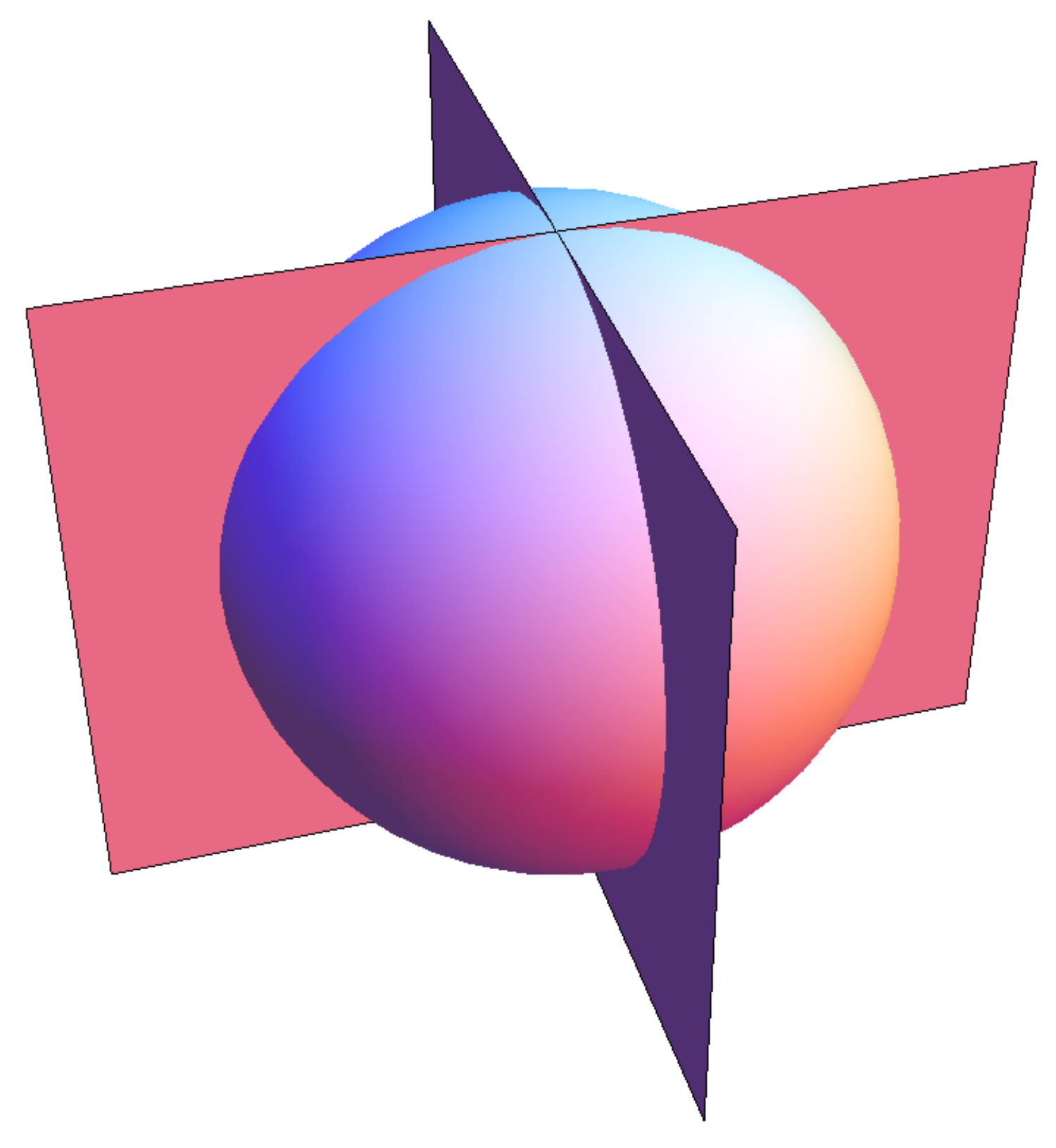}};
\end{scope}
\end{tikzpicture}
 \caption{The polytope $P$ (see Remark~\ref{rem:decoration} for the decoration), its reflection on the hyperplanes $K_{2,4}$ homeomorphic
to $Z$.}
\label{fk24}
\end{center}
\end{figure}
\end{ejm}

\begin{ejm}\label{ejm:sing-m=1}
This can be generalized as follows. Let $\multiconj:=((-1)^p, (1)^q, (0))$.
The polytope $P$
is the cone over $\Delta_{p-1} \times \Delta_{q-1}$ (its
intersection with the hyperplane $r_n=0$, recall Example~\ref{ejm:m=1-liso}
and
Proposition~\ref{prop:cono}) and apex in the $r_n$ axis, which is its only singular
point. In the same way, $Z$ is the suspension of $\mathbb{S}^{p-1}\times\mathbb{S}^{q-1}$.
\end{ejm}

We may consider also the case with several $0$'s; the polytope
will be obtained by iterated cones and the intersection by iterated suspensions.

\begin{remark}
For $p=q=2$, we obtain $P$ is the square pyramid and $Z$ is the suspension of the torus.
Note that in this case the polytope is not simple since one of the vertices fails to be simple.
\end{remark}

\begin{dfn}\label{def:sing}
We
say that $P$ is \emph{singular} at $F_S$
of \emph{depth}
$d_{F_S}:=\dim F_S-\#S+m+1$  (or $d_S$ for simplicity) if
$d_{F_S}>0$,
i.e., $\dim F_S\geq\#S-m$.
\end{dfn}

\begin{remark}
Under (WH), a continuous deformation of  the points $(A_1,\dots,A_n)$, as long as the condition (WH) is preserved, does not affect the combinatorics of $P$ nor the differential topology of $Z$, by Ehresmann type arguments. Without (WH), as far as a deformation is \emph{equisingular}, the topology of $Z$ does not change.

In terms of the polyhedron~$P$, \emph{equisingular} means that the combinatorics of $P$, including its relative
position with respect to the hyperplanes, does not change. In terms of $\multiconj$, \emph{equisingular} means
that the dimension of the spaces of proper convex combinations of the origin in $\R^m$ remain constant.
Note that the diffeomorphism class may change.
\end{remark}

\begin{remark}\label{rem:decoration}
Note that the singularities of a polytope~$P$ depend on its relative position with respect to the coordinate hyperplanes. In order to exhibit the singular faces,
we will decorate the faces with their depth.
For the sake of simplicity we may omit it whenever the depth
is the minimal possible one from the combinatorics of $P$, i.e.,
when the face is \emph{geometrically embedded}; in particular
for non-singular faces, the depth is zero and it will be usually omitted.
\end{remark}

\begin{ejm}
Let $m=1$ and $\multiconj=(-1,1,0^2)$. We obtain the polytope in Figure~\ref{fig:triang_d1}, where
the edge $F_{34}$ and its two vertices are singular of depth $1$.
It is not hard to check that $Z$ is homeomorphic to the union of two spheres with a common equator.

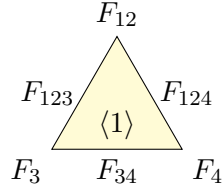
\begin{figure}[ht]
\centering
\begin{tikzpicture}
\filldraw[fill=white!80!yellow] (90:1) node[above] {$F_{12}$} -- node[left] {$F_{123}$}
(210:1)  node[below left] {$F_{3}$} -- node[below] {$F_{34}$}
node[above] {$\langle 1 \rangle$}
(-30:1) node[below right] {$F_{4}$} -- node[right] {$F_{124}$}
cycle;
\end{tikzpicture}
 \caption{Triangle with a singular edge}
\label{fig:triang_d1}
\end{figure}

\end{ejm}

\begin{dfn}\label{def:sing-isolated}
Let $v=F_S$ be a vertex of $P$. We say that
$P$ has an \emph{isolated singularity} at $v$
if $P$ is singular at~$v$ and not singular at the faces containing $v$.
We say that $P$ has a \emph{generic singularity}
at $v$ if moreover $0=\dim F_S=\#S-m$,
i.e. $\#S=m$  and $d_S=1$.
\end{dfn}

The tuple $\multiconj_S:=(A_j\mid j\in S)$ affinely generates
a subspace of dimension~$m-1$ and the origin is
in the interior of the convex closure of $\multiconj_S$.

\begin{ejm}
Let us consider a pyramid $P\subset\R^{n+1}$ with base an $n$-gon and such
that each face is the intersection with one coordinate hyperplane. For $n\geq 4$
it corresponds to an isolated singularity, only generic if $n=4$.
The intersections in Examples~\ref{ejm:m=1} and~\ref{ejm:sing-m=1} are generic singularities.
\end{ejm}

Let $Z$ be an intersection given by $\multiconj\subset(\R^m)^n$ with a generic
singularity. The singularity is given by $m$ points in $\R^m$ which have the origin as a proper convex combination and affinely generate a hyperplane (this is the case of Example~\ref{ejm:sing-m=1}). We may assume that they are $A_1,\dots, A_m$   and that they live in the hyperplane with last coordinate  $a_m =0$.  The other points are distributed in each of the two semispaces $a_m > 0$ ($p$ points) and $a_m < 0$ ($q$ points). The local structure is shown in the next result:

\begin{proposition}[{\cite[\S5.5.3]{LdM2023}}]
Let $Z\subset\R^n$ be an intersection of dimension $n-m-1$ defined by an $n$-tuple $\multiconj$ in $\R^m$
and determining a polytope~$P$.
A neighborhood of a generic singular vertex in $P$ is the cone over the product of two simplices
$\Delta^{p-1} \times \Delta^{q-1}$ ($n-m-1=p+q-1$).
The number of coordinate hyperplanes disjoint from the singular vertex is~$m$.
This singular vertex produces $2^m$ singularities homeomorphic to the the cone on a product of spheres $\mathbb{S}^{p-1}\times\mathbb{S}^{q-1}$.
\end{proposition}

We end this section with several examples which illustrate the topology.

\begin{ejm}
For the tuple $\multiconj$ in Figure~\ref{fig:one_generic} we have as $P$ a square with four
vertices $F_{14}$ (depth~$1$), $F_{135}$, $F_{235}$, and $F_{245}$. It is a non-geometric embedding
of $[0,1]^2\subset\R^2$. The manifold~$Z$
is the union of two tori with four points in common.

\begin{figure}[ht]
\centering
\begin{tikzpicture}
\coordinate (A1) at (0:1);
\coordinate (A2) at (75:1);
\coordinate (A3) at (105:1);
\coordinate (A4) at (180:1);
\coordinate (A5) at (-90:1);

\foreach \x in {1, ..., 5}
{
\fill (A\x) circle [radius=.05];
}

\draw
(A1) node[right] {$A_1$} --
(A2) node[right] {$A_2$} --
(A3) node[left] {$A_3$} --
(A4) node[left] {$A_4$} --
(A5) node[right] {$A_5$} --
cycle;
\draw (A1) -- (A3) -- (A5);
\draw (A4) -- (A2) -- (A5);
\draw (A1) -- (A4);
\fill[red] (0,0) circle [radius=.05];

\end{tikzpicture}
 \caption{Only one generic singularity in $P$}
\label{fig:one_generic}
\end{figure}
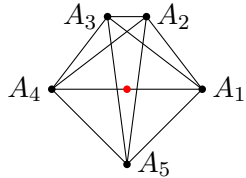

\end{ejm}

\begin{ejm}\label{ejm:k44}
Figure~\ref{fig:K44} illustrates an example with two generic singular vertices; the
intersection~$Z$ is homeomorphic to the complete bipartite graph $K_{4,4}$; each singular vertex
is associate to one partition.

\begin{figure}[ht]
\begin{subfigure}[b]{0.33\textwidth}
\centering
﻿\begin{tikzpicture}[scale=.9]
\cuadradob{0}{-5.5}{0,0}
\end{tikzpicture}
 \caption{$\mathcal{A}$}
\label{subfig:caso4c}
\end{subfigure}
\begin{subfigure}[b]{0.66\textwidth}
\centering
﻿\begin{tikzpicture}
\draw (2,0) node[below] {$r_2r_4$}-- (0,0) -- (0,2) node[above] {$r_1r_3$};

\draw (1.5,0) -- (0,1.5);
\fill[blue] (1.5,0) node[below left,black] {$F_{24}$} node[above right,black] {$\langle 1\rangle$} circle [radius=.1];
\fill[magenta] (0,1.5) node[ left,black] {$F_{13}$} node[right,black] {$\langle 1\rangle$} circle [radius=.1];
\begin{scope}[xshift=5cm, yshift=1cm]
\foreach \x [evaluate=\x as \i using int(\x+1)] in {0,...,7}
{
\coordinate (A\x) at (45*\x:1.25);
}

\foreach \x in {0,2,...,6}
\foreach \y in {1,3,...,7}
{
\draw (A\x) -- (A\y);
}
\foreach \x in {0,2,...,6}
{
\fill[blue] (A\x) circle [radius=.1];
}
\foreach \y in {1,3,...,7}
{
\fill[magenta] (A\y) circle [radius=.1];
}

\end{scope}
\end{tikzpicture}
 \caption{$P$ and $Z\cong K_{4,4}$.}
\label{subfig:K44}
\end{subfigure}
\caption{Two generic singularities}
\label{fig:K44}
\end{figure}
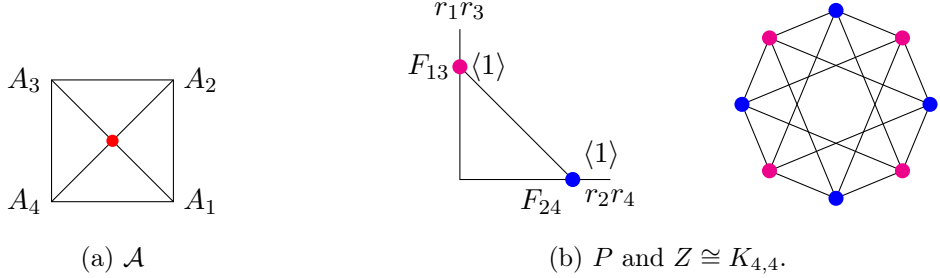

\end{ejm}

We can give a normalized description of all the tuples $\multiconj$ in $\R^2$ with only
generic singularities, compare with Example~\ref{ejm:m=2-liso}. We sketch the ideas of how to reach
the description stated in the next proposition. First, we can radially move
the points $A_i$ into the unit circle. 
In the unit circle we have the following allowed deformations.
If $A_{i_2}=-A_{i_1}$, we can deform both simultaneously without crossing
any other point; if $-A_{i_1}\notin\multiconj$ we can deform $A_{i_1}$ 
as $A_{i_1,t}$ ($A_{i_1,0}=A_{i_1}$) and keep fixed the other points, 
if $A_{{i_1},t}$ is never antipodal to any $A_j$, $j\neq i_1$.
Let $r$ be the number of generic singularities; there are
$r$ couples of antipodal points
and we can deform them such that they become equidistributed.
Since the singularities are generic these points cannot be repeated. 
Let $A_{i_1},A_{i_2}$ be two 
consecutive distinct points such that $-A_{i_1},-A_{i_2}\notin\multiconj$ and let $C(A_{i_1},A_{i_2})$ be the shorter arc
between them;
if there is no point of $\multiconj$ in the interior of the opposite of
$C(A_{i_1},A_{i_2})$ we can deform $A_{i_2}$ into $A_{i_1}$. 
After performing all the possible identifications, the next result
is straightforward though its proof is cumbersome.

\begin{proposition}
Let $\multiconj\subset\R^2$ with~$r$ generic singularities.
Then there exists an equisingular deformation such that
\[
\multiconj=(P_1,\dots,P_{2r})\cup\bigcup_{i=1}^{2r}
\underbrace{(Q_{i,1}^{n_{i,1}},\dots,Q_{i,m_i}^{n_{i,m_i}})}_{\multiconj_i}
\]
where
\begin{enumerate}[label=\rm(\alph{enumi})]
\item $(P_1,\dots,P_{2r})$ are equidistributed; in particular, $P_j+P_{r+j}=0$;
\item $\multiconj_i$ is an \emph{ordered} tuple between $P_i$ and $P_{i+1}$ ($P_{2r+1}:=P_1$);
\item $\abs{m_i-m_{i+r}}\leq 1$;
\item between two distinct consecutives points $Q',Q''\in\multiconj_i$, there exists exactly a unique point
$Q\in \multiconj_{i+r}$ such that $-Q\in C(Q',Q'')$; the
same statement holds interchanging $\multiconj_i$ and $\multiconj_{i+r}$.
\end{enumerate}
If $r=1$, then $m_1,m_2\geq 1$.
\end{proposition}

\begin{ejm}
The connected
intersections of dimension~$2$ with generic singularities are given by $n$-polygons
embedded in $\R^{n+r}$, $r\leq n$, where $r$ vertices of the polygon are in the intersection
of three coordinate hyperplanes.
Compare with the smooth case in Example~\ref{ejm:m=2-liso-total} where
the polygon is geometrically embedded.
\end{ejm}
 \section{Deformations and smoothings}\label{sec:smoothings}

Every singular intersection can be deformed slightly and become transversal. In most cases different slight deformations give intersections with different topological and differentiable types. We now study all the possibilities that can arise like this, first by considering all the possible slight deformations and then looking for lower dimensional deformations that include all the possible smooth types as well as the less singular ones that separate them.

Let $\multiconj=(A_1,\dots,A_n)\in(\R^m)^n$ satisfying \ref{N1},~\ref{N2}; let~$Z$ be  the intersection of ellipsoids in $\R^n$ they define and $P$ the corresponding polytope. For $\varepsilon >0$ let $U$ be the product of the balls $\mathbb{B}(A_i;\varepsilon)\subset\R^m$. Each $\multiconj_\mathbf{t}\in U$ defines an intersection $Z_\mathbf{t}$ and the corresponding polytope $P_\mathbf{t}$.

The open set $U$ admits a stratification~$\mathcal{S}$ such that in any stratum the embedded  type of $P_\mathbf{t}$  and, therefore, also the topological type of $Z_\mathbf{t}$ remains constant. The open strata will consist of tuples that satisfy (WH) and the differentiable type of $Z_\mathbf{t}$ and the combinatorial type of $P_\mathbf{t}$ remain constant.
A similar argument happens for the other strata, in general only for the topological type.

\begin{dfn}\label{def:smoothing}
With the above notation, we say that $U$ is the \emph{domain of a deformation} if $\multiconj$ is in the closure of all the strata
and  $\forall\mathbf{t}\in U$ the conditions \ref{N1} and \ref{N2} hold.

Each open stratum is called a  $\emph{smoothing}$ of $\multiconj$ and we will use this same name for the unique type of all the tuples of the stratum and their associated intersections and polytopes, since this cannot cause any confusion. A stratum $S$ where $\mathbf{0}\notin S$ and the $Z_\mathbf{t}$ is still singular for $\mathbf{t}\in S$ is called a \emph{partial smoothing}.

In many cases we will restrict to a subset $V\subset U$ that contains $\multiconj$ and the conditions about the strata hold for the induced stratification.
$V$ is called a  \emph{top-versal deformation} if all the topological types of smoothings in $U$ appear in~$V$.
\end{dfn}

\begin{remark}
Note that even if $\multiconj$ satisfies \ref{N3} it may happen that $\multiconj_\mathbf{t} \in U$ do not, i.e.,
$Z$ is connected and some $Z_\mathbf{t}$ are disconnected.
\end{remark}

\begin{ejm}\label{ejm:smoothings-1}
Let us consider $\multiconj:=((-1)^p, (1)^q, 0)\subset\R$. We can take as $U$ the product of intervals of radius $\varepsilon<\frac{1}{2}$, $p$ around the points $-1$, one around the point $0$, and $q$ around the points $-1$. There are three strata depending of the value of the parameter $t_0$ in the interval around $0$:
\begin{enumerate}[label=\rm(\roman{enumi})]
\item If $t_0<0$ we get an open smoothing stratum with intersection $\mathbb{S}^{p}\times\mathbb{S}^{q-1}$ and polytope $\Delta_p\times\Delta_{q-1}$.

\item If $t_0>0$ we get an open smoothing stratum with intersection $\mathbb{S}^{p-1}\times\mathbb{S}^{q}$ and polytope $\Delta_{p-1}\times\Delta_q$.

\item If $t_0=0$ we get a closed singular stratum of codimension one with intersection the suspension of $\mathbb{S}^{p-1}\times\mathbb{S}^{q-1}$ and polytope the cone on $\Delta_p\times\Delta_{q-1}$.
\end{enumerate}

The deformation $V\subset U$ where the parameter $t_0$ takes all values within its interval and all the other parameters are $0$ contains all those topological types and is therefore a top-versal deformation.  This situation generalizes to the case with only generic singularities.

The two possible smoothings  for $p=q=1$ are in~Figure~\ref{fig:smoothing-simple}.

%\tikzexternalenable
\begin{figure}[ht]
\begin{center}
\begin{tikzpicture}
\coordinate (O) at (0,0);
\coordinate (X) at (-135:{sqrt(2)/2});
\coordinate (Y) at (0:1);
\coordinate (Z) at (90:1.25);
\coordinate (T) at (.5,-.5);
\coordinate (P) at (0,.75);

\draw[->] (O) -- (X) node[left] {$r_1$};
\draw[->] (O) -- (Y) node[right] {$r_2$};
\draw[->] (O) -- (Z) node[above] {$r_3$};
\draw (P) -- (T);
\fill (T) circle [radius=.075];
\fill (P) circle [radius=.075];

\begin{scope}[xshift=5cm]
\coordinate (O) at (0,0);
\coordinate (X) at (-135:{sqrt(2)/2});
\coordinate (Y) at (0:1);
\coordinate (Z) at (90:1.25);
\coordinate (T) at (.5,-.5);
\coordinate (P) at (0.25,.75);

\draw[->] (O) -- (X) node[left] {$r_1$};
\draw[->] (O) -- (Y) node[right] {$r_2$};
\draw[->] (O) -- (Z) node[above] {$r_3$};
\draw (P) -- (T);
\fill (T) circle [radius=.075];
\fill (P) circle [radius=.075];
\end{scope}

\begin{scope}[xshift=-5cm]
\coordinate (O) at (0,0);
\coordinate (X) at (-135:{sqrt(2)/2});
\coordinate (Y) at (0:1);
\coordinate (Z) at (90:1.25);
\coordinate (T) at (.5,-.5);
\coordinate (P) at (-.25,.6);

\draw[->] (O) -- (X) node[left] {$r_1$};
\draw[->] (O) -- (Y) node[right] {$r_2$};
\draw[->] (O) -- (Z) node[above] {$r_3$};
\draw (P) -- (T);
\fill (T) circle [radius=.075];
\fill (P) circle [radius=.075];
\end{scope}
\end{tikzpicture}
 
﻿\begin{tikzpicture}[scale=2]
\begin{scope}[xshift=2.5cm, yshift=0.45cm, scale=.65]
\draw (-.1,1) -- (-3/2,0) -- (-.1, -1) -- (-1/2,0) --cycle
(.1,1) -- (3/2,0) -- (.1, -1) -- (1/2,0) --cycle;

\draw[line width=1.2, color=blue] ($.7*(-.1, 1)+.3*(-1/2, 0)$)
-- (-.1, 1) -- ($.7*(-.1, 1)+.3*(-3/2, 0)$)
($.7*(-.1, -1)+.3*(-1/2, 0)$)
-- (-.1, -1) -- ($.7*(-.1, -1)+.3*(-3/2, 0)$)
($.7*(.1, 1)+.3*(1/2, 0)$)
-- (.1, 1) -- ($.7*(.1, 1)+.3*(3/2, 0)$)
($.7*(.1, -1)+.3*(1/2, 0)$)
-- (.1, -1) -- ($.7*(.1, -1)+.3*(3/2, 0)$);
\end{scope}

\begin{scope}[xshift=5cm, yshift=0.45cm, scale=.65]

\draw (0,1) -- (1/2,0) -- (0, -1) -- (-1/2,0) --cycle
(0,1) -- (3/2,0) -- (0, -1) -- (-3/2,0) --cycle;

\draw[line width=1.2, color=blue] (0, 1) -- ($.7*(0, 1)+.3*(1/2, 0)$);
\draw[line width=1.2, color=blue] (0, 1) -- ($.7*(0, 1)+.3*(-1/2, 0)$);
\draw[line width=1.2, color=blue] (0, 1) -- ($.7*(0, 1)+.3*(3/2, 0)$);
\draw[line width=1.2, color=blue] (0, 1) -- ($.7*(0, 1)+.3*(-3/2, 0)$);

\draw[line width=1.2, color=blue] (0, -1) -- ($.7*(0, -1)+.3*(1/2, 0)$);
\draw[line width=1.2, color=blue] (0, -1) -- ($.7*(0, -1)+.3*(-1/2, 0)$);
\draw[line width=1.2, color=blue] (0, -1) -- ($.7*(0, -1)+.3*(3/2, 0)$);
\draw[line width=1.2, color=blue] (0, -1) -- ($.7*(0, -1)+.3*(-3/2, 0)$);
\end{scope}

\begin{scope}[xshift=7.5cm, yshift=0.45cm, scale=.65]

\draw (0,1.1) -- (-3/2,0) -- (0, -1.1) -- (3/2,0) --cycle
(0,.9) -- (1/2,0) -- (0, -.9) -- (-1/2,0) --cycle;

\draw[line width=1.2, color=blue] ($.7*(0,.9)+.3*(-1/2, 0)$)
-- (0,.9) -- ($.7*(0,.9)+.3*(1/2, 0)$)
($.7*(0,-.9)+.3*(-1/2, 0)$)
-- (0,-.9) -- ($.7*(0,-.9)+.3*(1/2, 0)$)
($.7*(0, 1.1)+.3*(-3/2, 0)$)
-- (0, 1.1) -- ($.7*(0, 1.1)+.3*(3/2, 0)$)
($.7*(0, -1.1)+.3*(-3/2, 0)$)
-- (0, -1.1) -- ($.7*(0, -1.1)+.3*(3/2, 0)$);

\end{scope}

\end{tikzpicture}
 
\includegraphics[height=3cm]{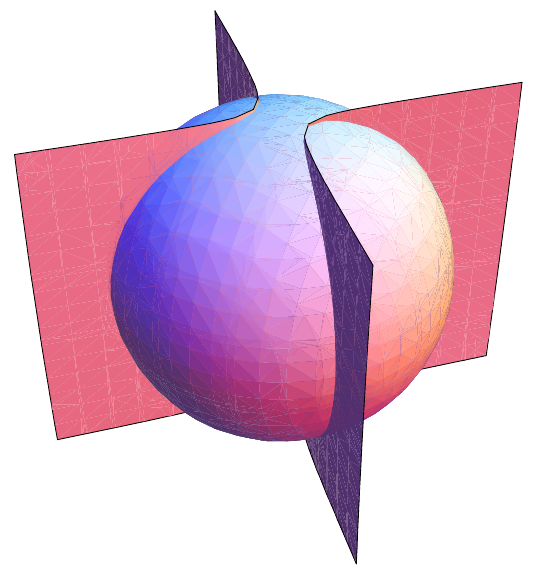}
\hspace{2cm}
\includegraphics[height=3cm]{ZdePFig1.pdf}
\hspace{2cm}
\includegraphics[height=3cm]{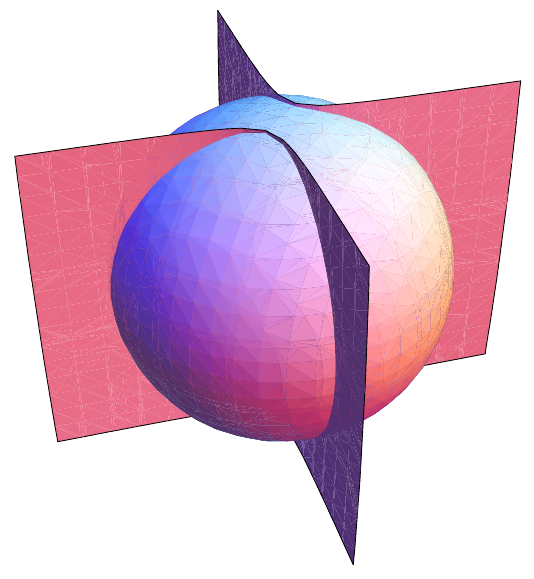}
\caption{Smoothings for $p=q=1$. In the first row, we have the polytopes. In the second row, an abstract representation of the smoothings, emphasizing what replace the singularities.
In the third row, the embedded intersections. The middle column is for the singular case, and the other columns for the smoothings.}
\label{fig:smoothing-simple}
\end{center}
\end{figure}
\end{ejm}

\begin{ejm}
Example~\ref{ejm:k44} has  generic singularities in dimension~$1$. A top-versal deformation has parameters in $\R^2$ with strata the origin, the four open semiaxes (whose intersections are two copies of the particular case of Example~\ref{ejm:smoothings-1} ($p=q=1$) and the four open quadrants (whose intersections are four copies of $\mathbb{S}^1$, i.e. $\mathbb{S}^1\times\mathbb{S}^0\times\mathbb{S}^0$).

Observe that here, as in the case $p=q=1$ of the Example~\ref{ejm:smoothings-1}, all the strata of positive dimension give disconnected intersections and only one topological type due to the symmetries of the examples. We pass from one smoothing to the other by surgeries.
\end{ejm}

\begin{ejm}
In dimension~$2$, the polytope associate to the simplest intersection with generic singularities is given by a triangle $\Delta_1$ with one vertex in the intersection
of three coordinate hyperplanes. Figure~\ref{fig:smoothing2-simple} shows the singular intersection and the two possible smoothings:
$\mathbb{S}^1\times\mathbb{S}^1$ and $\mathbb{S}^0\times\mathbb{S}^2$.
\begin{figure}[ht]
\begin{center}
﻿\begin{tikzpicture}[scale=1.2]
\begin{scope}[xshift=.25cm,yshift=.25cm]
\draw (.75,0) ellipse [x radius=.25cm,y radius=1cm];
\draw (.5,0) arc [start angle=180,end angle=360,x radius=.25,y radius=.125];
\draw[dashed] (.5,0) arc [start angle=180,end angle=0,x radius=.25,y radius=.125];
\draw[shift={(.5,0)}] (-.75,0) ellipse [x radius=.25cm,y radius=1cm];
\draw[shift={(.5,0)}] (-1,0) arc [start angle=180,end angle=360,x radius=.25,y radius=.125];
\draw[dashed, shift={(.5,0)}] (-1,0) arc [start angle=180,end angle=0,x radius=.25,y radius=.125];
\end{scope}

\begin{scope}[xshift=3.25cm,yshift=.25cm]
\draw[draw=black]
(0,-1) to[out=-30,in=-90] (1,0)
to[out=90,in=30] (0,1)
to[out=-30, in=90] (.5,0)
to[out=-90,in=30] (0,-1);
\draw (.5,0) arc [start angle=180,end angle=360,x radius=.25,y radius=.125];
\draw[dashed] (.5,0) arc [start angle=180,end angle=0,x radius=.25,y radius=.125];
\draw[draw=black]
(0,-1) to[out=-150,in=-90] (-1,0)
to[out=90,in=150] (0,1)
to[out=-150, in=90] (-.5,0)
to[out=-90,in=150] (0,-1);
\draw (-1,0) arc [start angle=180,end angle=360,x radius=.25,y radius=.125];
\draw[dashed] (-1,0) arc [start angle=180,end angle=0,x radius=.25,y radius=.125];
\end{scope}

\begin{scope}[xshift=6.25cm,yshift=.25cm]
\draw (0,0) circle [radius=.5];
\draw (0,0) circle [radius=1];
\draw (.5,0) arc [start angle=180,end angle=360,x radius=.25,y radius=.125];
\draw[dashed] (.5,0) arc [start angle=180,end angle=0,x radius=.25,y radius=.125];
\draw (-1,0) arc [start angle=180,end angle=360,x radius=.25,y radius=.125];
\draw[dashed] (-1,0) arc [start angle=180,end angle=0,x radius=.25,y radius=.125];
\end{scope}
\end{tikzpicture}
 \caption{}
\label{fig:smoothing2-simple}
\end{center}
\end{figure}
\end{ejm}

\begin{ejm}
Let us consider the intersection associated with a triangle~$\Delta_2$ with two generic singular vertices, see~Figure~\ref{fig:pentagonoZ}. The singular variety
$Z(\Delta_2)$ is in~\subref{subfig:pentagonoZf}; its smoothings are $\mathbb{S}^0\times\mathbb{S}^0\times\mathbb{S}^2$~\subref{subfig:pentagonoZa},
$\mathbb{S}^0\times\mathbb{S}^1\times\mathbb{S}^1$~\subref{subfig:pentagonoZb}, and $\mathcal{F}_5$~\subref{subfig:pentagonoZc}.

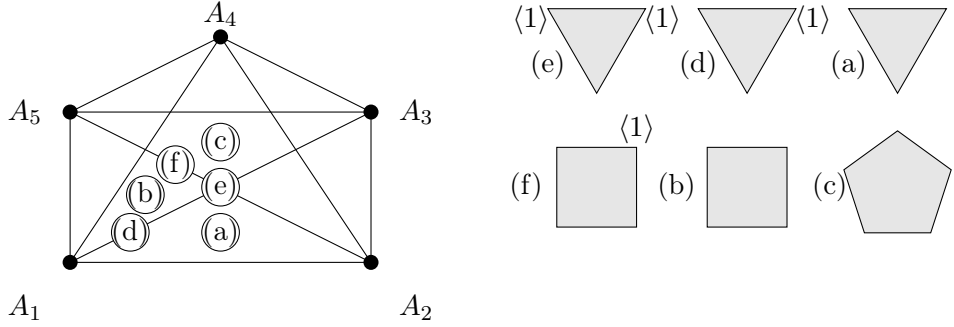
\begin{figure}
\centering
%\tikzexternaldisable
\begin{tikzpicture}
\coordinate (A1) at (-2,-2);
\coordinate (A2) at (2,-2);
\coordinate (A3) at (2,0);
\coordinate (A4) at (0,1);
\coordinate (A5) at (-2,0);
\coordinate (O1) at ($.25*(A1)+.25*(A2)+.25*(A3)+.25*(A5)$);
\coordinate (O2) at ($.65*(A5)+.35*(A2)$);
\coordinate (O3) at ($.8*(A1)+.2*(A3)$);
\coordinate (O4) at ($.65*(A1)+.15*(A3)+.2*(A4)$);
\coordinate (O5) at ($.45*(A1)+.45*(A2)+.2*(A4)$);
\coordinate (O6) at ($.2*(A1)+.2*(A2)+.2*(A3)+.4*(A4)+.2*(A5)$);

\foreach \x [evaluate=\x as \z using \x+1] in {1,...,5}
{
\fill (A\x) circle [radius=.1];
\node at ($1.3*(A\x)$) {$A_\x$};
}
\foreach \x [evaluate=\x as \z using \x+1] in {1,...,4}
\foreach \y in {\z,...,5}
{
\draw (A\x) -- (A\y);
}
\foreach \x in {1,...,6}
{
}

\filldraw[fill=white] (O1) node {\subref{subfig:pentagonoZf}} circle [radius=.25];

\filldraw[fill=white] (O2) node {\subref{subfig:pentagonoZe}} circle [radius=.25];

\filldraw[fill=white] (O3) node {\subref{subfig:pentagonoZd}} circle [radius=.25];

\filldraw[fill=white] (O4) node {\subref{subfig:pentagonoZb}} circle [radius=.25];

\filldraw[fill=white] (O5) node {\subref{subfig:pentagonoZa}} circle [radius=.25];

\filldraw[fill=white] (O6) node {\subref{subfig:pentagonoZc}} circle [radius=.25];

\begin{scope}[shift={(5,1)}, scale=.75]
\filldraw[fill=gray!20!white] (30:1) -- (150:1) -- (-90:1) -- cycle;
\node at (210:1) {\subref{subfig:pentagonoZf}};
\node at ($(30:1)+(0.3,-.2)$) {$\langle 1\rangle$};
\node at ($(150:1)-(0.3,.2)$) {$\langle 1\rangle$};

\end{scope}

\begin{scope}[shift={(5,-1)}, scale=.75]
\filldraw[fill=gray!20!white] (45:1) -- (135:1) -- (-135:1) -- (-45:1) -- cycle;
\node at (-1.25,0) {\subref{subfig:pentagonoZe}};

\node at ($(45:1)+(0,.325)$) {$\langle 1\rangle$};
\end{scope}

\begin{scope}[shift={(7,1)}, scale=.75]
\filldraw[fill=gray!20!white] (30:1) -- (150:1) -- (-90:1) -- cycle;
\node at (210:1) {\subref{subfig:pentagonoZd}};

\node at ($(30:1)+(0.3,-.2)$) {$\langle 1\rangle$};
\end{scope}

\begin{scope}[shift={(7,-1)}, scale=.75]
\filldraw[fill=gray!20!white] (45:1) -- (135:1) -- (-135:1) -- (-45:1) -- cycle;
\node at (-1.25,0) {\subref{subfig:pentagonoZb}};
\end{scope}

\begin{scope}[shift={(9,1)}, scale=.75]
\filldraw[fill=gray!20!white] (30:1) -- (150:1) -- (-90:1) -- cycle;
\node at (210:1) {\subref{subfig:pentagonoZa}};
\end{scope}

\begin{scope}[shift={(9,-1)}, scale=.75]
\filldraw[fill=gray!20!white] (18:1) -- (90:1) -- (162:1) -- (234:1) -- (306:1) -- cycle;
\node at (-1.2,0) {\subref{subfig:pentagonoZc}};
\end{scope}

\end{tikzpicture}
 %\tikzexternalenable
\caption{The left-hand side represents different $\multiconj$'s depending on whether the the origin in $\R^2$ is. The possibilities are the centers of the circles. The right-hand side represents
the polytopes for the intersections in Figure~\ref{fig:pentagonoZ}.}
\end{figure}

There are two partial smoothings, the square~$C_1$ with one generic singular vertex and a triangle $\tilde\Delta_2$ disjoint with one
coordinate hyperplane. The singular variety
$Z(C_1)$ is in~\subref{subfig:pentagonoZe}; its smoothings are in~\subref{subfig:pentagonoZb} and~\subref{subfig:pentagonoZc}.
Finally, $Z(\tilde{\Delta}_2)\cong\mathbb{S}^0\times Z(\Delta_1)$ is in~\subref{subfig:pentagonoZd}  with smoothings in~\subref{subfig:pentagonoZa} and~\subref{subfig:pentagonoZc}.

\begin{figure}[ht]
\centering
%\tikzexternaldisable
\begin{tikzpicture}[overlay, shift={(6.8,-.3)}]
\draw[->] (0,0) -- (0,1.25);
\fill[white] (0,.6) circle[radius=.2cm];
\draw[->] (0,0) -- (-4.5,1);
\draw[->] (0,0) -- (4.25,1);
\draw[->] (-1,-1.2) -- (-3,-1.2);
\draw[->] (1.2,-1.2) -- (3.5,-1.2);
\draw[->] (-5,0) -- (-5,1);
\draw[->] (-5,0) -- (-1,1.5);
\draw[->] (3.75,-.3) -- (3.75,1.25);
\draw[->] (3.75,-.3) -- (1,1.5);
\end{tikzpicture}
%\tikzexternalenable
\begin{subfigure}[b]{0.3\textwidth}
\centering
﻿\begin{tikzpicture}[scale=.5]
\esfera
\begin{scope}[xshift=2.5cm,xscale=-1]
\esfera
\end{scope}
\begin{scope}[xshift=2.5cm,xscale=-1]
\esfera
\end{scope}
\begin{scope}[yshift=-2.5cm,yscale=-1]
\esfera
\end{scope}
\begin{scope}[yshift=-2.5cm, xshift=2.5cm,scale=-1]
\esfera
\end{scope}
\end{tikzpicture}
 \caption{}
\label{subfig:pentagonoZa}
\end{subfigure}
\begin{subfigure}[b]{0.3\textwidth}
\centering
﻿\begin{tikzpicture}[scale=1]
\toro
\begin{scope}[xshift=.25cm,xscale=-1]
\toro
\end{scope}
\end{tikzpicture}
 \caption{}
\label{subfig:pentagonoZb}
\end{subfigure}
\begin{subfigure}[b]{0.3\textwidth}
\centering
﻿\begin{tikzpicture}[scale=.5]

\filldraw[fill=cyan!50!white] (0,0) circle [radius=2cm];
\filldraw[fill=white] (0,0) circle [radius=.375cm];
\filldraw[fill=white] (1.125,0) circle [radius=.375cm];
\filldraw[fill=white] (-1.125,0) circle [radius=.375cm];
\filldraw[fill=white] (0,1.125) circle [radius=.375cm];
\filldraw[fill=white] (0,-1.125) circle [radius=.375cm];
\draw (-2,0) arc [start angle=180,end angle=360,x radius=.25cm,y radius=.125cm];
\draw[dotted] (-2,0) arc [start angle=180,end angle=0,x radius=.25cm,y radius=.125cm];
\draw (-.75,0) arc [start angle=180,end angle=360,x radius=.1875cm,y radius=.1cm];
\draw[dotted] (-.75,0) arc [start angle=180,end angle=0,x radius=.1875cm,y radius=.1cm];
\draw (2,0) arc [start angle=0,end angle=-180,x radius=.25cm,y radius=.125cm];
\draw[dotted] (2,0) arc [start angle=0,end angle=180,x radius=.25cm,y radius=.125cm];
\draw (.75,0) arc [start angle=0,end angle=-180,x radius=.1875cm,y radius=.1cm];
\draw[dotted] (.75,0) arc [start angle=0,end angle=180,x radius=.1875cm,y radius=.1cm];
\draw (0,-2) arc [start angle=270,end angle=90,x radius=.125cm,y radius=.25cm];
\draw[dotted] (0,-2) arc [start angle=-90,end angle=90,x radius=.125cm,y radius=.25cm];
\draw (0,2) arc [start angle=90,end angle=270,x radius=.125cm,y radius=.25cm];
\draw[dotted] (0,2) arc [start angle=90,end angle=-90,x radius=.125cm,y radius=.25cm];
\draw (0,-.75) arc [start angle=270,end angle=90,x radius=.1cm,y radius=.1875cm];
\draw[dotted] (0,-.75) arc [start angle=-90,end angle=90,x radius=.1cm,y radius=.1875cm];
\draw (0,.75) arc [start angle=90,end angle=270,x radius=.1cm,y radius=.1875cm];
\draw[dotted] (0,.75) arc [start angle=90,end angle=-90,x radius=.1cm,y radius=.1875cm];

\end{tikzpicture} \caption{}
\label{subfig:pentagonoZc}
\end{subfigure}
\newline
\begin{subfigure}[b]{0.3\textwidth}
\centering
﻿\begin{tikzpicture}[scale=1]
\esfd
\begin{scope}[xshift=1.25cm,xscale=-1]
\esfd
\end{scope}
\begin{scope}[yshift=0cm,yscale=-1]
\esfd
\end{scope}
\begin{scope}[xshift=1.25cm,scale=-1]
\esfd
\end{scope}
\end{tikzpicture}
 \caption{}
\label{subfig:pentagonoZd}
\end{subfigure}
\begin{subfigure}[b]{0.3\textwidth}
\centering
﻿\begin{tikzpicture}[scale=1]
\esff
\begin{scope}[xshift=1cm,xscale=-1]
\esff
\end{scope}
\begin{scope}[yshift=0cm,yscale=-1]
\esff
\end{scope}
\begin{scope}[xshift=1cm,scale=-1]
\esff
\end{scope}
\end{tikzpicture}
 \caption{}
\label{subfig:pentagonoZf}
\end{subfigure}
\begin{subfigure}[b]{0.3\textwidth}
\centering
﻿\begin{tikzpicture}[scale=1]
\toro
\begin{scope}[xscale=-1]
\toro
\end{scope}
\end{tikzpicture}
 \caption{}
\label{subfig:pentagonoZe}
\end{subfigure}
\caption{}
\label{fig:pentagonoZ}
\end{figure}

\end{ejm}

\begin{remark}
Let $Z$ be an intersection with only generic singularities. We can define a top-versal deformation of $Z$ from the point of view of polytopes. The effect of this top-versal deformation
is to~\emph{simplify} the singular vertices. In dimension~$1$, a singular vertex is in the intersection of two coordinate hyperplanes, and the smoothings consist in moving the vertex to one of those hyperplanes. A similar discussion can be done in dimension~$2$. In dimension~$3$ there are two types of singular vertices. A vertex whose link is a quadrangle admits two smoothings
as in Figure~\ref{4pyra}. If it is a simple vertex in four coordinate hyperplanes, either the vertex moves to the intersection of three such hyperplanes (the polytope does not change and the intersection becomes multiplied by $\mathbb{S}^0$) or the polytope becomes truncated at this vertex.
\end{remark}

The topological types in adjacent open strata (separated by a codimension~$1$ stratum wall) are related by a move,
called \emph{flip} in \cite{B-M} where they applied it to describe the wall-crossing result in the context of moment-angle manifolds.

\begin{dfn}(\cite{B-M}) Let $P$ and $Q$ be two simple polytopes of the same dimension~$q$. Let $W$ be a simple polytope of dimension $q+1$. We say that $W$ is an \emph{elementary cobordism} between
$P$ and $Q$ if $P$ and $Q$ are disjoint facets of $W$ and $W\setminus (P \sqcup Q)$ contains a unique vertex $v$. Then, $Q$ is obtained from $P$ by a \emph{flip} of type $(a,b)$, where $a$ is the number of edges from $v$ to vertices in $P$ and $b$ is the number of edges from $v$ to vertices in $Q$.
\end{dfn}
%\tikzexternalenable

 %
 \section{\texorpdfstring{$3$}{3}-dimensional smoothings}
\label{sec:3smoothings}

%\tikzexternaldisable

In dimensions 1 and 2 it is possible to describe
all the types of intersections and their smoothings. In higher dimensions
this is not possible because there are too many polytopes.

In dimension 3 we will study the smoothings of geometrically embedded
polyhedra,
namely, some non-simple polytopes~$P\subset\R^3$ with vertices lying
at no more than $4$~faces. 

Let $\ell$ be
the number of vertices lying in four faces, called $4$-vertices in the sequel. There is a top-versal deformation
with parameter space homeomorphic to a neighbourhood of the origin in~$\R^\ell$, where the 
$2^\ell$ orthants correspond to the smoothings. 

For each orthant
there are $\ell$ possible $(2,2)$-flips. Note that the automorphism
group of the polytope acts on these sets and the number of
differentiable types of smoothings may be less that $2^\ell$.

We are going to associate to any such polytope a decorated graph.

\begin{dfn}\label{rem:grafo-cirugia}
Let $P$ a $3$-dimensional polytope with vertices lying
at no more than $4$~faces and geometrically embedded. The \emph{decorated graph} $\Gamma(P)$ of $P$
is defined as follows:
\begin{itemize}
\item The vertices are the orbits of the set of orthants under
the action of $(\mathbb{Z}/2)^\ell$.

\item The edges are the orbits 
of the $(2,2)$-flips associated with each non-simple vertex. 
\item The non-loop edges are arbitrarily oriented.

\item Each vertex is decorated with $[\alpha]$ where $\alpha$ is the size of the orbit.

\item Each non-loop edge is decorated with $(\beta_1,-\beta_2)$,
where $\beta_1$ is the number of the connecting flips
in the direction of the edge and $\beta_2$  is the number of the connecting flips
in the opposite direction of the edge.

\item The loop edges are decorated with $(\beta)$ where
$\beta$ is the number of flips which do not change the orbit.
\end{itemize}
\end{dfn}

Examples of these decorated graphs can be found in Figures~\ref{grafo:piramide}
and~\ref{fig:grafo-bp3}.

%\tikzexternalenable
\subsection{Smoothings of the square pyramid}
\mbox{}

Consider the square pyramid $P_{p4}$, the polytope associated to $\multiconj =((-1)^2,0,(1)^2)$ having a unique 4-vertex $V$.
Recall that $Z(P_{p4})$
has two singular points, the two apices of the suspension over the torus $\mathbb{S}^1\times \mathbb{S}^1$,
see Proposition~\ref{prop:cono}.
The neighbourhood of each of these singular points is a cone on the torus.  

The smoothings correspond to $\{\multiconj_t\mid t>0\}$ and $\{\multiconj_t\mid t<0\}$, where $\multiconj_t:=((-1)^2,t,(1)^2)$,
see Figure~\ref{4pyra} for the effect on the polytopes.
Observe that a new edge $E$ appears to replace the apex $V$.

\begin{figure}[ht]
\begin{center}
﻿\begin{tikzpicture}[scale=.7]
\cruz{0,0}{\node[left=3pt] at (0,0) {$V$};\node[left] at (-1,0) {$P_{p4}$};}
\draw[->] (0,1.2) -- (1.8,1.5);
\cruz{3,1.5}{\draw[line width=1.2,red] (0,.5) -- (0, -.5) ;}
\draw[->] (0,-1.2) -- (1.8,-1.5);
\cruz{3,-1.5}{\draw[line width=1.2,red] (.5,0) -- (-.5,0) ;}

\draw[->] (4.2,1.5) -- (4.8,1.5);

\begin{scope}[shift={(6,1.5)}]
\draw (-1,-1) rectangle (1,1);
\draw (-1,-1) -- (0,-.5) -- (1,-1);
\draw (-1,1) -- (0,.5) -- (1,1);
\draw[red, line width=1.2] (0,-.5) -- node[left,black] {$E$} (0,.5) ;
\end{scope}

\draw[->] (4.2,-1.5) -- (4.8,-1.5);

\begin{scope}[shift={(6,-1.5)}]
\draw (-1,-1) rectangle (1,1);
\draw (-1,-1) -- (-.5,0) -- (-1,1);
\draw (1,-1) -- (.5,0) -- (1,1);
\draw[red, line width=1.2] (-.5,0) --  node[above,black] {$E$} (.5,0) ;
\end{scope}
\end{tikzpicture}
 \caption{Smoothing at the 4-vertex V}
\label{4pyra}
\end{center}
\end{figure}
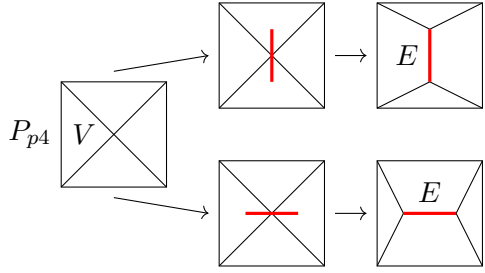

The group of symmetries of $P_{p4}$ is the dihedral group of order~$8$ and acts transitively on the two smoothings. 
Hence the corresponding polyhedron $P'_{p4}$ is a triangular prism, see Figure~\ref{4pyra},
with 5 faces and 6 vertices.
The new edge $E$ is a lateral edge of the prism. The manifold $Z(P'_{p4})$ obtained by reflection on the faces of $P'_{p4}$, is diffeomorphic to $\mathbb{S}^2\times\mathbb{S}^1$ which is the double of a solid torus.

The  edge $E$ is contained in two faces, hence under the reflections there are $2^3=8$ copies of $E$ in $Z(P'_{p4})$; it is disjoint to exactly one face, and as a consequence
they form two  circles, $\mathbb{S}^1_1=
\{p_1\}\times\mathbb{S}^1$ and $\mathbb{S}^1_2=
\{p_2\}\times\mathbb{S}^1$, in such a way that  $Z(P')$ is the union by the boundary of the two solid tori which are neighbourhood of these two circles; the solid tori are the product $\mathbb{D}_i\times \mathbb{S}^1$, where $\mathbb{D}_i$ is a closed disk centered at $p_i$; both disks
cover $\mathbb{S}^2$ and have common boundary.

The equations of the polytope defined by $\multiconj_t$ in $\R^5$ are
$0\leq r_i \leq 1$:
\begin{equation}\label{eq:pyramid}
\left\{
\begin{matrix}
r_1 &+& r_2 &+& r_3 &+& r_4 &+& r_5 & = & 1\\
r_1 &-& r_2 &+& r_3 &-& r_4 &+& tr_5 & = & 0.
\end{matrix}
\right.
\end{equation}
Let us consider the simple polytope $\Delta^2 \times \Delta^2$ of dimension $4$ in $\R^6$
defined by the equations
\[
\left\{
\begin{matrix}
r_1 &+& r_3 &+& \frac{1}{2}r_5 & = & \frac{1}{2}\\
r_2 &+& r_4 &+& \frac{1}{2}r_6 & = & \frac{1}{2},
\end{matrix}
\right.
\Leftrightarrow
\left\{
\begin{aligned}
r_1 + r_2 + r_3 + r_4 + \frac{1}{2}(r_5 + r_6) =&  1\\
r_1 - r_2 + r_3 - r_4 + \frac{1}{2}(r_5 - r_6) =& 0,
\end{aligned}
\right.
\qquad 0\leq r_i \leq 1.
\]
Let us cut it by a hyperplane $(1-t) r_5-(1+t)r_6=0$, $\abs{t}$ small enough. With the coordinates $r_1,\dots,r_4,r_5'$ in this hyperplane, where $r_5':=\frac{r_5+r_6}{2}$,
the equations of the cut become those of \eqref{eq:pyramid}.

\begin{remark}
The relation between the $2$ triangular prisms obtained by smoothing is a flip of type~$(2,2)$.
In a more topological language we pass from one copy to the other by a $0$-Dehn surgery on two fibers of the trivial fibration
$\mathbb{S}^2\times\mathbb{S}^1$.
\end{remark}

The following result is easy.

\begin{proposition} Let $\check{Z}(P_{p4})$ be the space obtained from $P_{p4}$ by reflections minus the points coming
from the non-simple vertex~$V$, which coincides with the space obtained from $P'_{p4}$ by reflections minus the circles coming from the edge~$E$. This space is homeomorphic to $\mathbb{S}^1\times\mathbb{S}^1\times(-1, 1)$.
In particular,
\begin{enumerate}[label=\rm(\alph{enumi})]
\item $\pi_1(Z(P_{p4}))=1$,
\item $\pi_1(Z(P'_{p4}))=\mathbb{Z}$,
\item $\pi_1(\check{Z}(P_{p4}))=\mathbb{Z}^2$.
\end{enumerate}
\end{proposition}

\begin{figure}[ht]
\centering
\begin{tikzpicture}[scale=.5]
\fill (-1,0) node[left] {$[2]$} circle[radius=.2cm];
\draw (0,0) circle[radius=1cm] node[right=5mm] {$(1)$};
\end{tikzpicture}
\caption{$\Gamma(P_{p4})$, see Definition~\ref{rem:grafo-cirugia}.}
\label{grafo:piramide}
\end{figure}
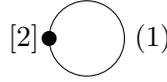

\subsection{Smoothing of triangular  bipyramid}
\mbox{}

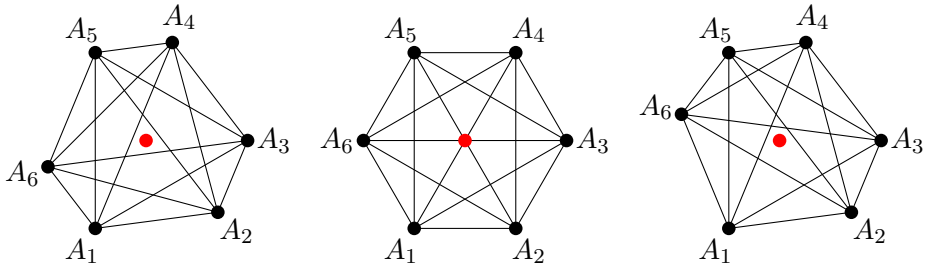
\begin{figure}[ht]
\centering
﻿\begin{tikzpicture}[scale=.9]

\begin{scope}[shift={(0,0)}]
\foreach \x in {1,3,5}
{
\coordinate (A\x) at ({60*(\x-3)}:1.5);
\fill (A\x) circle [radius=.1];
\node at ($1.26*(A\x)$) {$A_\x$};
}
\foreach \x in {2,4,6}
{
\coordinate (A\x) at ({60*(\x-3)+15}:1.5);
\fill (A\x) circle [radius=.1];
\node at ($1.26*(A\x)$) {$A_\x$};
}
\foreach \x [evaluate=\x as \y using \x+1] in {1,...,5}
\foreach \z in {\y,...,6}
{
\draw[name path=(L\x-\z)] (A\x) -- (A\z);
}
\fill[red] (0,0) circle [radius=.1];
\end{scope}
\end{tikzpicture}
 ﻿\begin{tikzpicture}[scale=.9]
\hexagono{0,0}
\end{tikzpicture}
 ﻿\begin{tikzpicture}[scale=.9]

\begin{scope}[shift={(0,0)}]
\foreach \x in {1,3,5}
{
\coordinate (A\x) at ({60*(\x-3)}:1.5);
\fill (A\x) circle [radius=.1];
\node at ($1.26*(A\x)$) {$A_\x$};
}
\foreach \x in {2,4}
{
\coordinate (A\x) at ({60*(\x-3)+15}:1.5);
\fill (A\x) circle [radius=.1];
\node at ($1.26*(A\x)$) {$A_\x$};
}

\coordinate (A6) at ({60*(3)-15}:1.5);
\fill (A6) circle [radius=.1];
\node at ($1.26*(A6)$) {$A_6$};

\foreach \x [evaluate=\x as \y using \x+1] in {1,...,5}
\foreach \z in {\y,...,6}
{
\draw[name path=(L\x-\z)] (A\x) -- (A\z);
}
\fill[red] (0,0) circle [radius=.1];
\end{scope}
\end{tikzpicture}
 \caption{$\multiconj\subset\R^2$ defining the equations of the bypiramid and the topological types of the smoothings.}
\label{fig:hexagono0}
\end{figure}

The polytope associated with $\multiconj\subset\R^2$
as in the center of Figure~\ref{fig:hexagono0} is the triangular bipyramid $BP_3$ in~$\R^6$. There are three singular $4$-vertices, each one with two possible smoothings, and therefore there exist $2^3=8$  simple polyhedra.
The bipyramid is defined by
$\multiconj:=(A_1,\dots,A_6)$, where $A_j:=\left(\cos\alpha_j,\sin\alpha_j\right)$, and $\alpha_j:=\frac{\pi (j - 3)}{3}$.

Note that deforming $A_1,A_2,A_4,A_5$ to be the vertices of the square of edge~$2$ centered at the origin,
the equation of the triangular bipyramid in $\R^6$ is
\begin{equation*}
\begin{matrix}
-r_1 &+& r_2 &+& r_3&+&r_4 &-& r_5 &-& r_6&=0\\
-r_1 &-&r_2 &+&     & &r_4 &+& r_5&&&=0\\
r_1  &+&r_2 &+& r_3&+&r_4&+&r_5&+&r_6&=1.
\end{matrix}
\end{equation*}

We can define a top-versal smoothing with parameter space $(-\varepsilon,\varepsilon)^2$, for some $0<\varepsilon\ll 1$, where
\[
\multiconj_{\mathbf{s}}\!:=\!(A_1, A_{2,s}, A_3, A_{4, s}, A_5, A_{6,t}),
\]
where $A_{2, s}:= \left(\cos(\alpha_2+s),\sin(\alpha_2\!+s)\right)$, 
$A_{4, s}:= \left(\cos(\alpha_4+s),\sin(\alpha_4+s)\right)$, and 
$A_{6, t}:= \left(\cos(\alpha_6+t),\sin(\alpha_6+t)\right)$, $s,t\in(-\varepsilon,\varepsilon)$.
This smoothing contains all the topological types. 

We come back to the general smoothing defined by \emph{moving} the $4$-vertices in $BP_3$,
for which we have give $2^3=8$ possible smoothings. The group of symmetries of $BP_3$ (isomorphic to $\mathfrak{S}_3\times\mathbb{Z}/2$)
acts on this space. This group is generated by the symmetry group of the intersection of the two pyramids and the exchange of the apices.

\begin{proposition}
Under the action of the symmetry group there are two orbits:
\begin{enumerate}[label=\rm(\roman{enumi})]
\item The cube $C$ ($2$ smoothings),
see Figure{\rm~\ref{bipyramid3a}};

\item the polyhedron $T_2$ ($6$ smoothings), see Figure{\rm~\ref{bipyramid3b}},
which is the \emph{book} with a pentagonal leaf or, equivalently, the double truncated tetrahedron.
\end{enumerate}

Any $(2,2)$-flip of $C$ produces $T_2$; two of the flips in a red edge of $T_2$ keep it, and the third one passes it to $C$.
\end{proposition}

\begin{proof}
In  Figure~\ref{bipyramid3a} (resp. Figure~\ref{bipyramid3b}) we have a smoothing
symmetric with a $3$-cycle in $\mathfrak{S}_3\subset\aut BP_3$ (resp. a 
transposition).
The two elements of the orbits of $C$ are the one in Figure~\ref{bipyramid3a} and the one with all the opposite smoothings. From Figure~\ref{bipyramid3b}, one can easily check how to find the
smoothings for $T_2$.
\end{proof}

  \begin{figure}[ht]
\begin{center}
\begin{tikzpicture}[scale=.7]
\newcommand\parte[1]{
\begin{scope}[rotate=#1]
\draw (0,0) -- (90:2) (90:1) -- (210:1);
\draw[dashed] (90:2) -- (90:3);
\draw[line width=1.2, red, shift={(90:1)}, rotate=75] (-1/2,0) -- (1/2,0);
\end{scope}
}
\parte{0}
\parte{120}
\parte{-120}

\begin{scope}[shift={(10,-.5)}]

\draw[line width=1.2, red] (0,0) node[black, below] {$1$} -- (2,0)
node[black, below] {$2$} (2,2) node[black, below left] {$3$} -- (2.5, 2.5) node[right, black] {$4$}
(.5, 2.5);

\draw[line width=1.2, red, dashed] (.5, 2.5) node[black, left] {$5$} -- (.5, .5) node[black, above right] {$6$};

\draw (2,0) -- (2,2) -- (0, 2) node[left] {$8$} -- (0,0)
(0,2) -- (.5, 2.5) -- (2.5, 2.5)
(2,0) -- (2.5, .5) node[right] {$7$} -- (2.5, 2.5);
\draw [dashed] (0, 0) -- (.5, .5) -- (2.5, .5);
\end{scope}

\begin{scope}[shift={(6,0)}]
\draw[line width=1.2, red, shift={(90:1)}, rotate=75] (-1/2,0) -- (1/2,0);
\draw[line width=1.2, red, shift={(-30:1)}, rotate=135] (-1/2,0) -- (1/2,0);
\draw[line width=1.2, red, shift={(210:1)}, rotate=195] (-1/2,0) -- (1/2,0);
\draw[dashed] (90:3) -- (90:2);
\draw (90:2) to[out=-90, in=75] ($(90:1)+(75:1/2)$) to ($(-30:1)+(135:1/2)$);
\draw[dashed] (210:3) -- (210:2);
\draw (210:2) to[out=30, in=195] ($(210:1)+(195:1/2)$) to ($(90:1)-(75:1/2)$);
\draw[dashed] (-30:3) -- (-30:2);
\draw (-30:2) to[out=150, in=-45] ($(-30:1)+(135:-1/2)$) to ($(210:1)-(195:1/2)$);
\draw (0,0) -- ($(90:1)-(75:1/2)$) (0,0) -- ($(210:1)-(195:1/2)$)
(0,0) -- ($(-30:1)-(135:-1/2)$);
\node at (20:2) {$0:7$};
\node at (5:2) {$\infty:8$};
\node[right] at ($(90:1)+(75:1/2)$) {$1$};
\node[left] at ($(90:1)-(75:1/2)$) {$2$};
\node[left] at ($(210:1)+(195:1/2)$) {$3$};
\node[below] at ($(210:1)-(195:1/2)$) {$4$};
\node[below] at ($(-30:1)+(135:-1/2)$) {$5$};
\node[right] at ($(-30:1)-(135:-1/2)$) {$6$};
\end{scope}

\end{tikzpicture}
 \caption{$C$-smoothing of the triangular bipyramid. The pictures on the left represent the bipyramid after stereographic projection where
the point at infinity is one of the simple vertices.}
\label{bipyramid3a}
\end{center}
\end{figure}
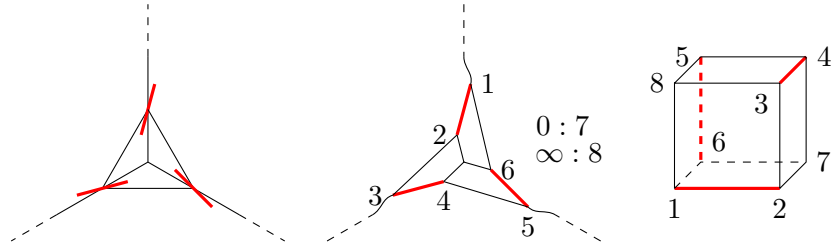

  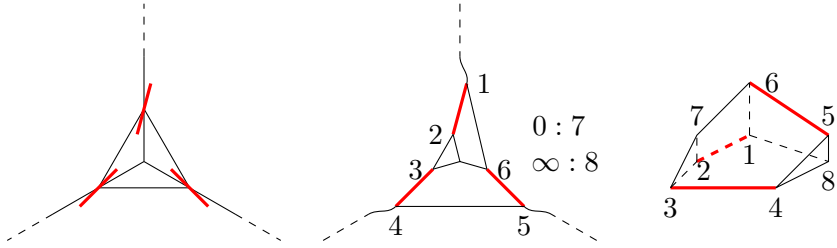
\begin{figure}[ht]
\begin{center}
\begin{tikzpicture}[scale=.7]
\newcommand\parteb[1]{
\begin{scope}[rotate=#1]
\draw (0,0) -- (90:2) (90:1) -- (210:1);
\draw[dashed] (90:2) -- (90:3);
\end{scope}
}
\parteb{0}
\parteb{120}
\parteb{-120}

\draw[line width=1.2, red, shift={(90:1)}, rotate=75] (-1/2,0) -- (1/2,0);
\draw[line width=1.2, red, shift={(-30:1)}, rotate=135] (-1/2,0) -- (1/2,0);
\draw[line width=1.2, red, shift={(210:1)}, rotate=225] (-1/2,0) -- (1/2,0);

\begin{scope}[shift={(10,-.5)}]

\draw[line width=1.2, red] (0,0) -- (2, 0) (3, 1) -- (1.5, 2);

\draw[line width=1.2, red, dashed] (.5, .5) -- (1.5, 1);

\draw  (2, 0) -- (3, 1)  (1.5, 2) node[right=2pt] {$6$} -- (.5,1) node[above] {$7$} --
(0,0) node[below] {$3$}
(2, 0) node[below] {$4$} -- (3, .5) -- (3, 1) node[above] {$5$};

\draw[dashed] (0,0) -- (.5, .5) node[below right=-4pt] {$2$} (1.5, 1) node[below] {$1$} -- (3, .5) node[below] {$8$} (.5, .5) -- (.5, 1)  (1.5, 1) -- (1.5, 2) ;

\end{scope}

\begin{scope}[shift={(6,0)}]
\begin{scope}
\draw[line width=1.2, red, shift={(90:1)}, rotate=75] (-1/2,0) -- (1/2,0);
\draw[line width=1.2, red, shift={(-30:1)}, rotate=135] (-1/2,0) -- (1/2,0);
\draw[line width=1.2, red, shift={(210:1)}, rotate=225] (-1/2,0) -- (1/2,0);
\draw[dashed] (90:3) -- (90:2);
\draw (90:2) to[out=-90, in=75] ($(90:1)+(75:1/2)$) to ($(-30:1)+(135:1/2)$);
\draw[dashed] (210:3) -- (210:2);
\draw (210:2) to[out=30, in=225] ($(210:1)+(225:1/2)$);
\draw ($(210:1)-(225:1/2)$) to ($(90:1)-(75:1/2)$);
\draw[dashed] (-30:3) -- (-30:2);
\draw (-30:2) to[out=150, in=-45] ($(-30:1)+(135:-1/2)$) to ($(210:1)+(225:1/2)$);
\draw (0,0) -- ($(90:1)-(75:1/2)$) (0,0) -- ($(210:1)-(225:1/2)$)
(0,0) -- ($(-30:1)-(135:-1/2)$);
\node at (20:2) {$0:7$};
\node at (0:2) {$\infty:8$};
\node[right] at ($(90:1)+(75:1/2)$) {$1$};
\node[left] at ($(90:1)-(75:1/2)$) {$2$};
\node[below] at ($(210:1)+(225:1/2)$) {$4$};
\node[left] at ($(210:1)-(225:1/2)$) {$3$};
\node[below] at ($(-30:1)+(135:-1/2)$) {$5$};
\node[right] at ($(-30:1)-(135:-1/2)$) {$6$};
\end{scope}

\end{scope}

\end{tikzpicture}
 \caption{$T_2$-smoothing of the triangular bipyramid.}
\label{bipyramid3b}
\end{center}
\end{figure}

The polyhedra have six bidimensional faces, one on each hyperplane $r_i=0$ and then the two obtained manifolds are connected. Using, e.g.~\cite{Gi-LdM}, we have that $Z(C)= \mathbb{S}^1\times \mathbb{S}^1\times \mathbb{S}^1=\mathbb{T}^3$.

Observe that the only simple polyhedra with 6 faces and 8 vertices are the cube $C$ and the double truncated tetrahedron $T_2$ and both are obtained in the smoothing of $BP_3$.

To obtain the manifold $Z(T_2)$ we use the following result.

\begin{proposition}[{\cite[Theorem~2.1]{Gi-LdM}}, cf. also \cite{B-M}]\label{prop:truncar}
Given a simple polytope $P$ of dimension~$d$, let $P_{\text{\rm sf}}(v)$ be the
the result of truncating the vertex $v$ and let $n$ be the number
of facets.
Then
$Z(P_{\text{\rm sf}}(v))$ is diffeomorphic to the connected sum
\[
Z(P) \# Z(P) \# (2^{n-d}-1) (\mathbb{S}^{d-1} \times \mathbb{S}^1).
\]
\end{proposition}

\begin{remark}
Observe that the topology of $Z(P_{\text{\rm sf}}(v))$ does not depend on which vertex of $P$ is truncated. This gives infinitely many examples of different $P$ with the same $Z(P)$, in particular, the result of the operation $P_{\text{\rm sf}}(v)$ on a simplex successively
$n$ times is independent of the vertices chosen.
Curiously, in any dimension we get the same number of copies of
$\mathbb{S}^{d-1} \times \mathbb{S}^1$ since it only depends on $m-d$. This includes the genera of the surfaces associated to the polygons.
\end{remark}

Let $\Delta$ be a tetrahedron and let $v, w$ be two vertices.
Then,
\[
Z(\Delta_{\text{\rm sf}}(v)) \cong
Z(\Delta) \# Z(\Delta)\# (2^{4-3} -1)(\mathbb{S}^2 \times \mathbb{S}^1)
\cong \mathbb{S}^2 \times \mathbb{S}^1
\]
and
\[
Z(T_2) = Z((\Delta_{\text{\rm sf}}(v))_{\text{\rm sf}}(w)) \cong
\mathbb{S}^2 \times \mathbb{S}^1 \# \mathbb{S}^2 \times \mathbb{S}^1
\# (2^{5-3} -1)(\mathbb{S}^2 \times \mathbb{S}^1)\cong
5\# (\mathbb{S}^2 \times \mathbb{S}^1).
\]
The manifolds $Z(C)$ and $Z(T_2)$ are tessellated by $2^6=64$ copies of $C$ or $T_2$ respectively.
Recall that the $(2,2)$-flips are $0$-surgeries on the following link.
Each one of the three red edges $E_i$ in the cube $C$ replacing a singular vertex of the bipyramid $BP_3$ give rise to a link $L_i\subset Z(C)$ of $2^2$ components, each of these components is made by $2^2$ copies of $E_i$ and is one fiber of the 3-dimensional torus.
Observe that a tubular neighbourhood of $L_i$ has a \emph{canonical longitude} defined by a parallel fiber or by the intersection with the copies in $Z(C)$  of one of the two  faces of $C$ containing $E_i$. Then 0-surgery in $L_i\subset Z(C)$ produces the manifold $Z(T_2)$. We have proved the following result.

\begin{proposition}\label{prop:pZdebipiramide}
 The result of $0$-surgery on a  link composed by $4$ parallel fibers of the $3$-dimensional torus $\mathbb{T}_3$ is the manifold $5\# (\mathbb{S}^2 \times \mathbb{S}^1)$
\end{proposition}

Analogously, a $0$-surgery in the $4$-component link in $Z(T_2)$ defined by the red edge belonging to a pentagonal face  in $T_2$ produces the manifold $Z(C)$, meanwhile  $0$-surgery in the $4$-component link in $Z(T_2)$ defined by the red edge belonging to a square face  in $T_2$ produces again $Z(T_2)$.

The graph $\Gamma(BP_3)$ is in Figure~\ref{fig:grafo-bp3}: recall that the weights of the edges refer both to the number of $(2,2)$-flips in the polytopes and to the number of equivariant
$0$-surgeries in the link of the manifold.

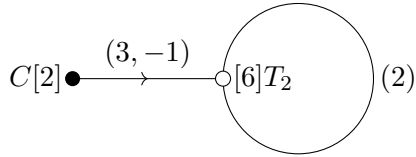
\begin{figure}[ht]
\centering
\begin{tikzpicture}

\coordinate (A) at (0,0);
\coordinate (B) at (2,0);
\fill (A) node[left] {$C[2]$} circle [radius=.1];
\draw[->-=.5] (A) -- node[above] {$(3,-1)$} (B);
\draw (B) arc [start angle=-180,end angle=180,radius=1];
\filldraw[fill=white] (B) node[right] {$[6]T_2$} circle [radius=.1] node[right=1.95cm] {$(2)$};
\end{tikzpicture}
 \caption{Decorated graph $\Gamma(BP_3)$.}
\label{fig:grafo-bp3}
\end{figure}

 \subsection{Hyperbolic bipyramid}\label{subsec:hbp}
\mbox{}

There exits a finite volume hyperbolic triangular bipyramid whose three $4$-vertices are ideal points and whose dihedral angles
are $\frac{\pi}{2}$. This result  follows from Andreev's Theorem (\cite{and1970b}, \cite[Chap. 6, Theorem 2.8]{alekseevskij-vinberg-sl:93}). We give here a construction of such hyperbolic triangular bipyramid.

\begin{figure}[ht]
\centering
\begin{tikzpicture}[scale=1.75]
\coordinate (A) at (0,0);
\coordinate (B) at (0,1);
\coordinate (C) at (1,0);
\coordinate (D) at (-135:1);
\coordinate (E) at (1/4,1/8);
\draw (A) node[left] {$A$} -- ($1.2*(B)$) (A) -- ($1.2*(C)$) (A) -- ($1.2*(D)$);
\draw[line width=1] (E) node[above right] {$E$} -- (B) node[above left] {$B$} (E) -- (C) node[below right] {$C$} (E) -- (D) node[left] {$D$} (B) -- (C) -- (D) --cycle;
\end{tikzpicture}
 \caption{}
\label{fig:hbp}
\end{figure}

  \begin{proposition}\label{prop:hbp}
There is a bipyramid $BP_H$ in the three-dimensional hyperbolic space $\mathbb{H}^3$ where all the $4$-vertices are ideal.
\end{proposition}

\begin{proof}
We use the Klein model of the  hyperbolic three-dimensional space $\mathbb{H}^3$. It consists in the open ball $\mathbb{B}^3$ of radius 1 in the three-dimensional Euclidean space $\mathbb{E}^3$. To every $k$-dimensional plane $\Pi\subset \mathbb{H}^3$ there corresponds a Euclidean plane $\widetilde{\Pi}\subset\mathbb{E}^3$ such that $\Pi=\widetilde{\Pi}\cap B^3$. This model is not conformal and the hyperbolic angle between two planes $\Pi_1$ and $\Pi_2$ is the Euclidean angle between their boundary  circles $\widetilde{\Pi}_1\cap \overline{B^3}$ and  $\widetilde{\Pi}_2\cap \overline{B^3}$.

Let us consider a triangular bipyramid in $\mathbb{B}^3$  with $4$-vertices at the points $D=(1,0,0)$, $C=(0,1,0)$, and $B=(0,0,1)$ and simple vertices at $A=(0,0,0)$ and $E=(a,a,a)$ for some $a\in\left(\frac{1}{3},\frac{\sqrt{3}}{3}\right)$.

It is clear that the dihedral angles at the edges  $AB$, $AC$ and $AD$ are $\frac{\pi}{2}$  because the intersections of the coordinate planes and the unit sphere are circles that meet orthogonally. We can compute $a$ in order to have also $\frac{\pi}{2}$ dihedral angles at the edges $BC$, $CD$, and $BD$. The edge $CD$ is the intersection
$ACD\cap ECD$.
The plane $\widetilde{\Pi} _{ECD}$ containing $ECD$, and hence the points $D=(1,0,0)$, $C=(0,1,0)$ and $E=(a,a,a)$, is
given by the equation
   \[
   ax+ay+(1-2a)z=a.
   \]
The plane $\widetilde{\Pi} _{ACD}$ is $ z=0$.
The intersection $\widetilde{\Pi} _{ACD} \bigcap \widetilde{\Pi}_{ECD}$ is a line intersecting $\partial \overline{\mathbb{B}}^3$ at the points $(1,0,0)$ and $(0,1,0)$. The angle  between the circles
$C_1=\widetilde{\Pi} _{ACD}\bigcap \partial \overline{\mathbb{B}}^3$ and $C_2=\widetilde{\Pi} _{ECD}\bigcap \partial \overline{\mathbb{B}}^3$ is the dihedral angle in $CD$.
The tangent vector $v_1$ to $C_1$ in $D=(1,0,0)$ is proportional to the cross product
of the normal vector of the plane $\widetilde{\Pi} _{ACD}$ and the normal vector to the tangent plane to $\partial \overline{B}^3$ at the point $D=(1,0,0)$:
\[
   v_1=(a,a,1-2a) \times (1,0,0)=(1,1-2a,-a)
   \]
Analogously the tangent vector $v_2$ to $C_2$ in $D=(1,0,0)$ is the cross product
   \[
   v_2=(0,0,1) \times (1,0,0)=(0,1,0)
   \]
Therefore the formula for the scalar product  and the condition on the angle to be $\frac{\pi}{2}$ allows us to compute $a$:
\[
0 = \langle v_1, v_2\rangle= \langle (1,1-2a,-a), (0,1,0)\rangle= 1-2a =0 \Rightarrow a=\frac{1}{2}
\]
By symmetry, the dihedral angles in  $BC$, and $BD$ is also $\frac{\pi}{2}$.
Therefore
\begin{alignat*}{4}
\widetilde{\Pi}_{ECD}:& \frac{1}{2}x+\frac{1}{2}y &= \frac{1}{2} &\Leftrightarrow x+y&=1\\
\widetilde{\Pi}_{EBC}:& \frac{1}{2}y+\frac{1}{2}z &= \frac{1}{2} &\Leftrightarrow y+z&=1\\
\widetilde{\Pi}_{EBD}:& \frac{1}{2}x+\frac{1}{2}z &= \frac{1}{2} &\Leftrightarrow x+z&=1.
\end{alignat*}
The angle between $\widetilde{\Pi} _{ECD}$ and $\widetilde{\Pi} _{EBC}$ at the point $C=(0,1,0)$ is computed as before.
\begin{align*}
v_{\widetilde{\Pi}_{ECD}}=(1,1,0)\times (0,1,0)&= (0,0,1)\\
v_{\widetilde{\Pi}_{EBC}}=(0,1,1)\times (0,1,0)&= (1,0,0)\\
\langle v_{\widetilde{\Pi}_{ECD}}, v_{\widetilde{\Pi}_{EBC}}\rangle=\langle(0,0,1),(1,0,0)\rangle&=0
\end{align*}
By symmetry again, the angle in all the others edges is also $\frac{\pi}{2}$.
\end{proof}

There is a tessellation $\mathfrak{T}_{BP}$ by these bipyramids in $\mathbb{H}^3$.
Let $G_{BP}$ the subgroup of $\aut\mathbb{H}^3$ generated by the hyperbolic reflections on the six planes containing the faces  of $BP_H$
which is the automorphism group of $\mathfrak{T}_{BP}$. The reflections corresponding to facets sharing an edge commute as they are orthogonal.

\subsection{Groups and the bipyramid}
\mbox{}

The $2$-skeleton of the tessellation by bipyramids is formed by hyperbolic planes tessellated by triangles with two ideal vertices, and the $1$-skeleton are geodesics.
The quotient $\mathbb{H}^3/G_{BP}$ defines a hyperbolic orbifold structure $\mathbf{BP}$ in $BP_H$.

Let us consider the natural epimorphism $\omega_{BP}:G_{BP}\twoheadrightarrow(\mathbb{Z}/2)^6$. Recall that $BP_3$ has a natural orbifold
coming from the fact that $BP_3$ is the quotient of $Z(BP_3)$ under the action of $(\mathbb{Z}/2)^6$ by reflections on the coordinate hyperplanes
of $\R^6$.

The existence of this natural epimorphism and the orbifold structure holds in general for any intersection $Z$ and
its associated polytope $P$: since $P$ is the quotient of $Z$ by an action
of $(\mathbb{Z}/2)^n$, it acquires a natural orbifold structure. If a point is
contained in $k$ coordinate hyperplanes, its isotropy group is $(\mathbb{Z}/
2)^k$.
The orbifold fundamental group of
$P$ is
\[
\pi_1^{\text{orb}}(P) =
\langle a_1,\dots,a_n\mid
a_i^2=1, [a_i,a_j]=1\text{ if } P\cap\{x_i=x_j=0\}\neq\emptyset
\rangle,
\]
see \cite{HDu,H:90} for the notion of \emph{orbifold fundamental group}.
In the case of a geometric embedding of a polytope with $n$ facets in $\R^n$, the generators
are in bijection with the facets and two generators commute if and only if the two facets have non-empty intersection. The case of simple polytopes
was studied in~\cite[Lemma~4.4]{DJ1991}; in this case two generators commute if and only if the facets intersect in a codimension~$2$ face.
The fundamental group of $Z$ is the kernel of the natural epimorphism $\pi_1^{\text{orb}}(P)\twoheadrightarrow(\mathbb{Z}/2)^n$.

For the case of geometrically embedded polyhedra of dimension~$3$, let $\check{P}$ be the
complement of the non-simple vertices in~$P$ and let $\check{Z}(P)\subset Z$ be its preimage under the quotient map.
Note that $\check{P}$ inherits a natural structure of orbifold for which
\[
\pi_1^{\text{orb}}(\check{P}) =
\langle a_1,\dots,a_n\mid
a_i^2=1, [a_i,a_j]=1\text{ if } \check{P}\cap\{x_i=x_j=0\}\neq\emptyset
\rangle.
\]
The following result is straightforward from Proposition~\ref{prop:hbp}.

\begin{coro}
The two orbifold structures over $\textbf{BP}$ and $\check{BP}_3$ are isomorphic. In particular, the manifold $\check{Z}(BP_3)$ admits a complete hyperbolic geometric structure
with $3\times 2^2$ ends of torus type,
$G_{BP}\cong\pi_1^{\text{\rm orb}}(\check{BP}_3)$, and $\check{Z}(BP_3)=\ker\omega_{BP}$ the derived subgroup.
\end{coro}

Let us compute some of these fundamental groups. Note that  $\check{Z}(BP_{3})$ coincides with the $3$-torus $Z(C)$ minus the $12$~circles coming from three pairwise disjoint edges one in each direction.

\begin{proposition}\label{prop:group_bipiramide}
\mbox{}
\begin{enumerate}[label=\rm(\alph{enumi})]
\item\label{Ga} $\pi_1(Z(BP_{3}))=1$,
\item\label{Gb} $\pi_1(Z(C))=\mathbb{Z}^3$,
\item\label{Gc} $\pi_1(Z(T_2))=\mathbb{F}_5$, the free group with $5$~generators,
\item\label{Gd} $\pi_1(\check{Z}(BP_{3}))$ is the kernel of the map $G\to(\mathbb{Z}/2)^3$,
\[
G:=\langle x, y, z\mid [x,[y,z]]=[y,[z,x]]=[z,[x,y]]=1\rangle,
\]
where the map sends the generators to generators; the abelianized of this group is $\mathbb{Z}^{12}$.
Actually $\check{Z}(BP_{3})$ is a $(\mathbb{Z}/2)^3$-Galois cover
of the complement of three pairwise disjoint factors
of $\mathbb{T}^3$.
\end{enumerate}
\end{proposition}

\begin{proof}
The cases \ref{Gb} and \ref{Gc} are trivial. For the case \ref{Ga},
since all the facets of $BP_3$ intersect, then $\pi_1^{\text{orb}}(BP_3)\cong(\mathbb{Z}/2)^6$ and the kernel is trivial.

For \ref{Gd}, we may use $\pi_1^{\text{orb}}(\check{BP_3})$, but we are going to
use an alternative construction: $\check{Z}(BP_3)$ is
a $(\mathbb{Z}/2)^3$-Galois cover of the complement of three pairwise disjoint factors
of $\mathbb{T}^3$, whose fundamental group is~$G$.
\end{proof}

We finish the discussion
with the notion of \emph{small cover}, see~\cite{DJ1991}, i.e., the \emph{smallest} cover of $\check{BP}_3$ which
is a manifold (in this case a complete hyperbolic manifold). This small cover is obtained from an epimorphism
$\pi_1^{\text{orb}}(BP_3)\to(\mathbb{Z}/2)^2$, sending a generator of associated to a facets to a non-trivial
element of $(\mathbb{Z}/2)^2$, such that the images of generators associated to intersecting faces are distinct.
This small cover has $3$ ends of torus type. It is not hard to check that the fundamental
group of this small cover is isomorphic to
\[
\langle a_1,\dots,a_4\mid
a_1^2=[a_2, a_1\cdot a_3]=[a_3, a_1\cdot a_4]=[a_4, a_1\cdot a_2]=1
\rangle.
\]
 \section{The octahedron variety}
\label{sec:octahedron}

The octahedron is a non-simple polyhedron with only $4$-vertices. The study
of its smoothings is richer than in the previous examples and it deserves its own section.
As for the other ones, there are models of this polyhedron in $\R^8$ where the faces are the intersection
with the coordinate hyperplanes and, in particular, the faces are pairwise orthogonal.
In the following section we give another presentation of an octahedron with orthogonal
faces in the hyperbolic space.

\subsection{The octahedron \texorpdfstring{$P_O$}{P0}}
\mbox{}

We can construct a geometric embedding $P_O\subset\R^8$ of the  octahedron  given by the equations
$r_i\geq 0$, $\sum_{i=1}^8 r_i =1$, and
\begin{align*}
r_{1} - r_{5} - r_{6} - r_{7} + 2 r_{8} &=0 &
r_{2} - r_{5} - r_{6} + r_{8} &=0\\
r_{3} - r_{5} - r_{7} + r_{8}&=0 &
r_{4} - r_{6} - r_{7} + r_{8}&=0,
\end{align*}
hence
\[
\multiconj =
\left\{
\begin{pmatrix}
 1 \\
 0 \\
 0 \\
 0
\end{pmatrix},
\begin{pmatrix}
 0 \\
 1 \\
 0 \\
 0
\end{pmatrix},
\begin{pmatrix}
 0 \\
 0 \\
 1 \\
 0
\end{pmatrix},
\begin{pmatrix}
 0 \\
 0 \\
 0 \\
 1
\end{pmatrix},
\begin{pmatrix}
 -1 \\
 -1 \\
 -1 \\
 0
\end{pmatrix},
\begin{pmatrix}
 -1 \\
 -1 \\
 0 \\
 -1
\end{pmatrix},
\begin{pmatrix}
 -1 \\
 0 \\
 -1 \\
 -1
\end{pmatrix},
\begin{pmatrix}
 2 \\
 1 \\
 1 \\
 1
 \end{pmatrix}
 \right\}.
\]
We can construct the intersection of ellipsoids $Z(P_O)$ which is a $3$-dimensional variety with singularities. In the sequel we study
the smoothings of this singular variety. We are also interested in the smooth part $\check{Z}(P_O)$, which is topologically obtained as
the union of the reflections of $\check{P}_O$, the octahedron minus its six vertices. Since each vertex is disjoint to $4$~facets
$Z(P_O)\setminus\check{Z}(P_O)$ consists of $96=2^4\times 6$ points; all the smoothings will admit a submanifold homeomorphic to
$\check{Z}(P_O)$ which will be the complement of a link with $96$ components, and we  connect the different smoothings
by $0$-surgeries of sub-links with $16$ components.

\subsection{Hyperbolic geometry of the octahedron orbifold}
\mbox{}

As in \S\ref{subsec:hbp}, there exists a hyperbolic ideal octahedron of finite volume with $\frac{\pi}{2}$ dihedral angles. In fact, it is known, \cite[pg. 217]{vinberg-sh:93}, that the only regular polyhedra with a hyperbolic structure with $\frac{\pi}{2}$ dihedral angles are the dodecahedron and the octahedron. The dodecahedron is a bounded hyperbolic polyhedron and the octahedron is ideal with finite volume. The hyperbolic dodecahedron was used in \cite{ALdML:16} to study a smooth intersection of quadrics but the octahedron correspond to a singular intersection of ellipsoids.
The next result is a explicit construction of the hyperbolic structure in the octahedron.

\begin{proposition}\label{prop:oh}
There is a regular octahedron $O_H$ in the three-dimensional hyperbolic space $\mathbb{H}^3$ where all the vertices are ideal, hence,
$O_H$ is homeomorphic to $\check{P}_O$.
\end{proposition}

\begin{proof}
Consider the regular octahedron $P_O$ with vertices $(\pm 1,0,0)$, $(0,\pm 1,0)$, $(0,0,\pm 1)$ sitting in  a ball $\mathbb{B}_R^3$ of radius $R$ centered at the origin. This octahedron is symmetric respect to the coordinate planes. The interior of the ball $\mathbb{B}^3_R$ is taken as the Klein model, $\mathbb{H}^3$, for the hyperbolic space. We want to choose
$R$ such that  every dihedral angle defined
by two adjacent faces in $P_O$ is $\frac{\pi}{2}$. The dihedral angle between two intersecting planes in the Klein model is the same as the Euclidean angle between
their bounding circles in~$\partial \mathbb{B}^3_R$.

\begin{figure}[ht]
\begin{center}
\begin{tikzpicture}[line join=bevel,z=-5.5,scale=2]
\coordinate (A1) at (0,0,-1);
\coordinate (A2) at (-1,0,0);
\coordinate (A3) at (0,0,1);
\coordinate (A4) at (1,0,0);
\coordinate (B1) at (0,1,0);
\coordinate (C1) at (0,-1,0);

\draw[dashed] (A1) node[above] {$D$} -- (A2)  node[below left] {$(0,-1,0)$} node[above left] {$E$} -- (B1)
 node[above left] {$(0,0,1)$} node[above right] {$A$}-- cycle (A4)  node[below right] {$(1,0,0)$} node[above right] {$B$} -- (A1) --  (C1)  node[below left] {$(0,0,-1)$} node[below right] {$F$};

\draw [line width= 1, fill opacity=0.7,fill=yellow!30!white] (A3) -- (A4) -- (B1) -- cycle;
\draw [line width= 1, fill opacity=0.7,fill=yellow!30!white] (A2) -- (A3) -- (C1) -- cycle;
\draw [line width= 1, fill opacity=0.7,fill=yellow!30!white] (A3) -- (A4) -- (C1) -- cycle;
\draw [line width= 1, fill opacity=0.7,fill=yellow!30!white] (A2) -- (A3)  node[below left] {$(0,1,0)$} node[above left] {$C$} -- (B1) -- cycle;

\draw[gray] ($1.2*(A2)-.2*(A4)$) -- ($1.2*(A4)-.2*(A2)$);
\draw[gray] ($1.2*(B1)-.2*(C1)$) -- ($1.2*(C1)-.2*(B1)$);
\draw[gray] ($1.4*(A1)-.4*(A3)$) -- ($1.4*(A3)-.4*(A1)$);

\node at (-3,1) {Face $1$: $ABC$};
\node at (-3,.75) {Face $2$: $ABD$};
\node at (-3,.5) {Face $3$: $ADE$};
\node at (-3,.25) {Face $4$: $ACE$};
\node at (-3,0) {Face $5$: $BDF$};
\node at (-3,-.25) {Face $6$: $BCF$};
\node at (-3,-.5) {Face $7$: $CEF$};
\node at (-3,-.75) {Face $8$: $DEF$};
\end{tikzpicture}
 \caption{Octahedron}
\label{f1}
\end{center}
\end{figure}
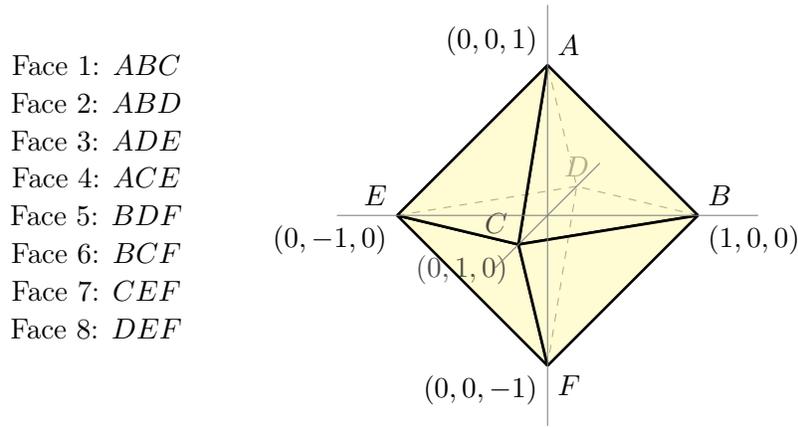

Let us compute the dihedral angle between the face $ABC$ and the face $ACE$. The vectors $(1,1,1)$ and $(1,-1,1)$ are respectively perpendicular to the planes $ABC$ and $ACE$. The bounding circles $ABC\cap \partial \mathbb{B}^3_R$ and $ACE\cap \partial \mathbb{B}^3_R$ intersects in the point $(x_0,0,z_0)$, where $x_0+z_0=1$ and $x_0^2+z_0^2=R^2$. The vector $(x_0,0,z_0)$ is  perpendicular to the sphere $\partial \mathbb{B}^3_R$ at $(x_0,0,z_0)$, so that
\(
(x_0,0,z_0)\times (1,1,1)= (-z_0, z_0-x_0,x_0)
\)
is tangent to the bounding circle $ABC\cap \partial \mathbb{B}^3_R$. Similarly, the vector
\(
(x_0,0,z_0)\times (1,-1,1)= (z_0, z_0-x_0,-x_0)
\)
is tangent to the bounding circle $ACE\cap \partial \mathbb{B}^3_R$.
Thus, $R$ should be chosen so that
\begin{align*}
  \langle(-z_0, z_0-x_0,x_0), (z_0, z_0-x_0,-x_0)\rangle = -2x_0 z_0&=0 \\
  x_0+z_0&=1 \\
  x_0^2+z_0^2 &= R^2.
\end{align*}
The first two equations have two solutions for $(x_0,z_0)$, $\{(0,1),(1,0)\}$, the two intersecting points of the bounding circles. In both cases, the third equation implies $R=1$.
Hence we have proved that the hyperbolic regular octahedron with right dihedral angles has ideal vertices and it is inscribed in $\partial \mathbb{B}^3_1$.
This regular octahedron $O_H$ is homeomorphic to $\check{P}_O$.
\end{proof}

The faces of $O_H$ are contained in hyperbolic planes and the edges are geodesics. In fact, the group $G_O$ generated by reflection on the eight planes containing the faces acts on $\mathbb{H}^3$ and it produces a regular tessellation $\mathfrak{T_{P_H}}$ made up by regular octahedra. The reflections corresponding to facets sharing an edge commute as they are orthogonal.

\subsection{Groups and octahedron}
\mbox{}

The $2$-skeleton of the tessellation by octahedra is formed by hyperbolic planes tessellated by ideal triangles. The $1$-skeleton is formed by geodesics and all the vertices are ideal points. The quotient $\mathbb{H}^3/G_O$ defines a hyperbolic orbifold structure $\mathbf{O}$ in $O_H$.
There is a natural epimorphism $\omega_H:G_O\twoheadrightarrow(\mathbb{Z}/2)^8$. On the other side,
$\check{P}_O$ has a natural orbifold structure from the quotient $\rho_O:\check{Z}(P_O)\to\check{P}_O$, which is associated to the monodromy
$\omega_O:\pi_1^{\text{\rm orb}}(\check{P}_O)\to(\mathbb{Z}/2)^8$ of $\rho_O$.
The following result comes from Proposition~\ref{prop:oh}.

\begin{coro}
The two orbifold structures over $O_H$ and $\check{P}_O$ are isomorphic. In particular,  $\check{Z}(P_O)$ admits a complete hyperbolic geometric structure,
$G\cong\pi_1^{\text{\rm orb}}(\check{P}_O)$, and $\check{Z}(P_O)=\ker\omega$ the derived subgroup.

The group $\pi_1^{\text{\rm orb}}(\check{P}_O)$ has $8$ generators $x_1,\dots,x_8$, labelled as the facets in Figure{\rm~\ref{f1}} with the relations
$x_i^2=1$, $i=1,\dots,8$, and
\begin{equation}\label{eq:group}
\begin{aligned}
{[x_i, x_j]} = 1,\ (i, j)\in \{&(1,2),(2,3),(3,4),(4,1),(5,6),(6,7),\\
&(7,8),(8,5),(1,6),(2,5),(3,8),(4,7)\}
\end{aligned}
\end{equation}
corresponding to the pairs of facets having an edge in common.
The group $\pi_1(\check{Z}(P_O))$ is the derived subgroup of $\pi_1^{\text{\rm orb}}(\check{P}_O)$ and its abelianization is $\mathbb{Z}^{100}$.
\end{coro}

The assertion about $\pi_1(\check{Z}(P_O))$ follows from the general arguments in the proof of Proposition~\ref{prop:group_bipiramide}.
A presentation of $\pi_1(\check{Z}(P_O))$ and the rank of its abelianization can be found in \url{https://github.com/enriqueartal/SingularQuadricIntersections}.
There are several intermediate orbifold covers defined by epimorphisms $\tau:\pi_1(\check{Z}(P_O))\twoheadrightarrow(\mathbb{Z}/2)^m$, $1\leq m\leq 8$, between $\check{Z}(P_O)$ and $\check{P}_0$ which are interesting.

In order to check the \emph{orbifold points} of these covers we must study how they behave for the different isotropy groups:
\begin{enumerate}[label=\rm(I\arabic{enumi})]
\setcounter{enumi}{-1}
\item The isotropy group of a point $p$ in the interior of $P_O$ is trivial, $P_O$ is a manifold around~$p$ and no condition is needed for the covering to be a manifold around
its preimages.
\item\label{I1} The isotropy group of a point $p$ in the interior of a facet of $P_O$ is $\mathbb{Z}/2$ and is generated by some $x_i$. Hence the covering associated to $\tau$
is a manifold around
the preimages of~$p$ if $\tau(x_i)\neq 0$.
\item\label{I2} The isotropy group of a point $p$ in the interior of an edge of $P_O$ is $(\mathbb{Z}/2)^2$ and is generated by some $x_i, x_j$. Hence the covering associated to $\tau$
is a manifold around
the preimages of~$p$ if $\tau(\{x_i,x_j\})$ is a subgroup isomorphic to $(\mathbb{Z}/2)^2$.

\item There is no point whose isotropy group $(\mathbb{Z}/2)^3$ as there is no simple vertex.

\item Though there is no point whose isotropy group $(\mathbb{Z}/2)^4$, this is the case if we consider $P_O$ instead of $\check{P}_O$. In this case
$p$ is a vertex, $V$ is a closed regular neighborhood of $p$ in $P_O$, then the
orbifold fundamental group of the punctured neighbourhood $\check{V}:=V\setminus\{p\}$ is
\begin{gather*}
G_{i,j,k,l}:=\langle
x_{i}, x_{j}, x_{k},x_l;\, x_{i}^{2}, x_{j}^{2}, x_{k}^{2}, x_{l}^{2}, (x_{i}x_{j})^2,
(x_{j}x_{k})^2, (x_{k}x_{l})^2,(x_{l}x_{i})^2 \rangle,\\
(i,j,k,l)\in\{(1,2,3,4), (5,6,7,8), (1,2,5,6), (3,4,7,8), (1, 4, 6, 7), (2, 3, 5, 8)\}.
\end{gather*}
The derived subgroup is isomorphic $\mathbb{Z}^2$.
Let us study the connected components $\check{W}$ of the preimage of this neighbourhood in the cover induced by~$\tau$.

\begin{enumerate}[label=\rm(V\arabic{enumii})]
\item If $\tau(G_{i,j,k,l})$ is $\mathbb{Z}/2$ and the condition in \ref{I1} holds for $i,j,k,l$, then the boundary of $\check{W}$ is homeomorphic to
the orbifold $\mathbb{S}^2_{2222}$.

\item If $\tau(G_{i,j,k,l})$ is $(\mathbb{Z}/2)^2$ and the condition in \ref{I2} holds
for $(i,j)$, $(j,k)$, $(k,l)$, $(l,i)$, then the boundary of $\check{W}$ is homeomorphic to
a $2$-torus.

\item If $\tau(G_{i,j,k,l})$ is $(\mathbb{Z}/2)^4$, then the boundary of $\check{W}$ is homeomorphic to
a $2$-torus.
\end{enumerate}
\end{enumerate}

\begin{proposition}
Let $M_\tau$ be the total space of the orbifold cover associated to $\tau:\pi_1(\check{Z}(P_O))\twoheadrightarrow(\mathbb{Z}/2)^m$.

\begin{enumerate}[label=\rm(\alph{enumi})]
\item If $m=1$ and \ref{I1} holds, then $M_\tau$
is the \emph{double} of $P_O$ as a complete hyperbolic orbifold homeomorphic to the sphere $\mathbb{S}^3$ minus 6 points.

\begin{enumerate}[label=\rm(\roman{enumii})]
\item There are six cusps, coming from the vertices.
\item There are~$12$ geodesics coming from the
$1$-skeleton of the octahedron for which the angle around the edges is $\pi$.
\item The boundary of a neighbourhood of every cusp is the orbifold $\mathbb{S}^2_{2222}$, a two dimensional sphere with~$4$ cone points with isotropy the cyclic group of order~$2$.
\end{enumerate}

\item If $\tau$ is defined as
%\tikzexternaldisable
\[
\begin{tikzcd}[row sep=0pt]
\pi_1(\check{Z}(P_O))\rar&(\mathbb{Z}/2)^2\\
x_1,x_3,x_6,x_8\rar[mapsto]&e_1\\
x_2,x_4,x_5,x_7\rar[mapsto]&e_2,
\end{tikzcd}
\]
then $M_\tau$ is a complete hyperbolic manifold with $6$ ends which are of torus type
and $\pi_1(M_\tau)$ is isomorphic to
\[
\langle a_1,\dots,a_6\mid [a_1,a_2]=[a_3,a_4]=[a_5,a_4 a_1]=[a_6,a_2a_3]=1,
a_4 a_6 a_5 a_3 =a_5a_4a_3a_6
\rangle.
\]

\end{enumerate}
\end{proposition}

The hyperbolic manifold $\check{Z}(P_O)$ has $96$ cusps. The variety $Z(P_O)$ is obtained
adding points in the places of the cusps. Following the ideas in the proof
of Proposition~\ref{prop:group_bipiramide} and computations with \texttt{Sagemath}
we get some topological information on $Z(P_O)$.

\begin{proposition}\label{prop:groups-octa}
The orbifold fundamental group of $P_O$ has the generators and relations of $\pi_1^{\text{\rm orb}}(\check{P}_O)$
plus the relations
\begin{equation}\label{eq:group1}
\begin{aligned}
{[x_i, x_j]} = 1,\ (i, j)\in \{&(1,3),(2,4),(5,7),(6,8),(1,7),(4,6)\\
&(1,5),(2,6),(3,7),(4,8),(3,8),(4,7),(3,5),(2,8)\}.
\end{aligned}
\end{equation}
and the fundamental group of $Z(P_O)$ is isomorphic to $\mathbb{Z}^4$.
\end{proposition}

\subsection{The smoothings of the intersection associated to the octahedron}
\mbox{}

It is useful to work with a compact manifold $Z_T$ which has the homotopy type of $\check{Z}(P_O)$
which is obtained by removing conic neighborhoods of the singular points of $Z(P_O)$. This
is manifold whose boundary is composed by~$96$ tori. Actually, since these tori are the boundary
of neighborhoods of generic singularities for each boundary component we have
 a natural product decomposition $\mathbb{S}^1\times\mathbb{S}^1$ on each boundary component.
Hence we can obtain a smooth manifold by adding a solid torus $\mathbb{D}^2\times\mathbb{S}^1$ along each boundary component and compatible
with the product structure. It corresponds to compactify  $\check{Z}(P_O)$ by $\mathbb{S}^1$'s instead of points.
There are two possible ways to add this solid tori at each boundary component. Since the cusps are related by the monodromy action of the covering, the boundary components come in $6$ packages of $24$ tori and then we have $64=2^6$ natural choices, which coincide with the possible smoothings of $Z(P_O)$.

Since the symmetry group of $P_O$ acts on these choices, as in the previous sections they are distributed in orbits and we will have to study less cases.
Each choice corresponds to a polyhedron with $8$~faces (coming from the faces of the octahedron) $12+6$ edges (coming from the edges of the octahedron and from its vertices)
and $2\cdot 6$ vertices (each vertex of the octahedron is doubled).

\begin{proposition}
There are exactly~$14$ simply connected simple polyhedra of dimension~$3$ with exactly $8$~faces, $18$~edges,
and $12$~vertices (\cite{alloctahedra, Dutch20}). Up to permutation, only for six of them it is possible to contract six pairwise
disjoint edges to obtain an octahedron, but one of them in two distinct ways.
\end{proposition}

The second part of this proposition is related with the result we are interested in: up to automorphism
there are exactly seven distinct topological types of smoothings. There is an apparent contradiction
between these two statements. As we can see in Figure~\ref{fig:al}, one of the smoothings
corresponds to the same polyhedron but there are two distinct ways of contracting onto a octahedron.
The details can be found in \url{https://github.com/enriqueartal/SingularQuadricIntersections}.

\begin{thm}
The pairs composed by the polyhedra and the sets of red edges in Figure{\rm~\ref{fig:al}}
are the seven topological types of smoothings of the octahedron.
The size of the orbits of each type are:
$12$ in {\rm\subref{subfig:al2}}, $24$ in {\rm\subref{subfig:al3}}, $2$ in {\rm\subref{subfig:al4}}, $4$ in {\rm\subref{subfig:al5}}, $4$ in {\rm\subref{subfig:al6}},
$6$ in {\rm\subref{subfig:al1}}, $12$ in {\rm\subref{subfig:al1a}}.
\end{thm}

%\tikzexternalenable
\begin{figure}[ht]
\begin{subfigure}[b]{0.3\textwidth}
\centering
﻿\begin{tikzpicture}[scale=1]

\begin{scope}[xshift=3cm]
\foreach \a in {0,...,4}
{
\coordinate (A\a) at (90*\a + 45:1);
\coordinate (B\a) at (90*\a + 45:2/5);
}
\foreach \a[evaluate=\a as \b using \a+1] in {0,...,3}
{
\draw (B\a) -- (B\b);
}

\coordinate (A21) at ($.8*(A2) + .2*(A1)$);
\coordinate (A23) at ($.8*(A2) + .2*(A3)$);
\coordinate (A30) at ($.8*(A3) + .2*(A0)$);
\coordinate (A32) at ($.8*(A3) + .2*(A2)$);
\coordinate (A22) at ($1/3*(A2)+2/3*(B2)$);
\coordinate (A33) at ($1/3*(A3)+2/3*(B3)$);

\draw (A0) -- (A1) -- (A21) -- (A23) -- (A32)
-- (A30) -- cycle;
\draw (A0) -- (B0);
\draw (A1) -- (B1);
\draw (A21) -- (A22) -- (A23);
\draw (A30) -- (A33) -- (A32);
\draw (A22) -- (B2);
\draw (A33) -- (B3);

\draw[red, line width=1] (A32) -- (A23)
(B2) -- (A22) (B3) -- (A33) (A4) -- (A30)
(A1) -- (A21) (B0) -- (B1);
\end{scope}
\end{tikzpicture}
 \caption{}
\label{subfig:al2}
\end{subfigure}
\begin{subfigure}[b]{0.3\textwidth}
\centering
﻿\begin{tikzpicture}[scale=1]

\begin{scope}[xshift=3cm]
\foreach \a in {0,...,5}
{
\coordinate (A\a) at (72*\a + 90:1);
\coordinate (B\a) at (72*\a+ 90:.5);
}
\foreach \a [evaluate=\a as \b using \a+1] in {0,...,4}
{
\draw (A\a) --(A\b);
}
\foreach \a in {1,...,4}
{
\draw (A\a) -- (B\a);
}
\coordinate (A00) at ($.5*(A0) + .5*(B0)$);
\coordinate (B01) at ($.5*(B0) + .5*(B1)$);
\coordinate (B04) at ($.5*(B0) + .5*(B4)$);

\draw (A0) -- (A00);
\draw (B01) -- (B04);
\draw (A00) -- (B01) -- (B1) -- (B2) --
(B3) -- (B4) -- (B04) --cycle;

\draw[red, line width=1] (B01) -- (B1)
(B04) -- (B4) (B2) -- (B3) (A1) -- (A2)
(A0) -- (A00) (A3) -- (A4);
\end{scope}
\end{tikzpicture}
 \caption{Scutoid}
\label{subfig:al3}
\end{subfigure}
\begin{subfigure}[b]{0.3\textwidth}
\centering
\begin{tikzpicture}[scale=.7]

\foreach \a in {0,...,3}
{
\coordinate (A\a) at (90 + \a*120:1/3);
}
\foreach \a[evaluate=\a as \b using \a+1] in {0,...,2}
{
\draw (A\a) -- (A\b);
}

\foreach \a in {0,...,3}
{
\coordinate (B\a) at ([shift={(90:4/3)}, rotate=60]90 + \a*120:1/3);
}
\foreach \a[evaluate=\a as \b using \a+1] in {0,...,2}
{
\draw (B\a) -- (B\b);
}

\foreach \a in {0,...,3}
{
\coordinate (C\a) at ([shift={(210:4/3)}, rotate=60]90 + \a*120:1/3);
}
\foreach \a[evaluate=\a as \b using \a+1] in {0,...,2}
{
\draw (C\a) -- (C\b);
}

\foreach \a in {0,...,3}
{
\coordinate (D\a) at ([shift={(-30:4/3)}, rotate=60]90 + \a*120:1/3);
}
\foreach \a[evaluate=\a as \b using \a+1] in {0,...,2}
{
\draw (D\a) -- (D\b);
}

\draw (A0) -- (B1);
\draw (A1) -- (C2);
\draw (A2) -- (D3);

\draw (B0) -- (C0);
\draw (C1) -- (D1);
\draw (D2) -- (B2);

\draw[red, line width=1] (A0) -- (B1)
(D2) -- (B2) (A2) -- (D0) (A1) -- (C2)
(C0) -- (B0) (D1) -- (C1);
\end{tikzpicture}
 \caption{4-truncated tetrahedon}
\label{subfig:al4}
\end{subfigure}

\begin{subfigure}[b]{0.225\textwidth}
\centering
\begin{tikzpicture}[scale=1.5]

\coordinate (A0) at (3/5, -2/5);
\coordinate (A1) at (1/2, 1/2);
\coordinate (A2) at (-1/2, 1/2);
\coordinate (A3) at (-3/5, -3/5);
\coordinate (A4) at (2/5, -3/5);
\coordinate (A5) at (A0);

\coordinate (B0) at (1/4, -1/4);
\coordinate (B1) at (1/4, 1/4);
\coordinate (B2) at (0, 1/4);
\coordinate (B3) at (-1/4, 0);
\coordinate (B4) at (-1/4, -1/4);
\coordinate (B5) at (B0);

\coordinate (C0) at (2/5,-2/5);
\coordinate (C2) at (-1/3,1/3);

\foreach \a[evaluate=\a as \b using \a+1] in {0,...,4}
{
\draw (A\a) -- (A\b);
\draw (B\a) -- (B\b);
}

\draw (A0) -- (C0) -- (B0);
\draw (C0) -- (A4);
\draw (A1) -- (B1);
\draw (A2) -- (C2) -- (B2);
\draw (C2) -- (B3);
\draw (A3) -- (B4);

\draw[red, line width=1] (B1) -- (B2)
(C2) -- (A2) (A0) -- (A1) (C0) -- (B0)
(B3) -- (B4) (A4) -- (A3);

\end{tikzpicture}
 \caption{D\"{u}rer-Solid}
\label{subfig:al6}
\end{subfigure}
\begin{subfigure}[b]{0.225\textwidth}
\centering
\begin{tikzpicture}

\foreach \a in {0,...,6}
{
\coordinate (A\a) at (\a*60:1);
\coordinate (B\a) at (\a*60:1/2);
}
\foreach \a[evaluate=\a as \b using \a+1] in {0,...,5}
{
\draw (A\a) -- (A\b);
\draw (B\a) -- (B\b);
\draw (A\a) -- (B\a);
}

\draw[red, line width=1] (B3) -- (B4)
(B1) -- (B2) (B0) -- (B5)
(A4) -- (A5) (A0) -- (A1) (A2) -- (A3);

\end{tikzpicture}
 \caption{Hexagonal prism}
\label{subfig:al5}
\end{subfigure}
\begin{subfigure}[b]{0.225\textwidth}
\centering
﻿\begin{tikzpicture}

\begin{scope}[xshift=2.5cm]
\foreach \a in {0,...,5}
{
\coordinate (A\a) at (72*\a + 90:1);
\coordinate (B\a) at (72*\a+ 90:.5);
}
\foreach \a [evaluate=\a as \b using \a+1] in {0,...,4}
{
\draw (A\a) --(A\b);
\draw (B\a) --(B\b);
}
\foreach \a in {0,1,4}
{
\draw (A\a) -- (B\a);
}
\coordinate (C2) at ($.8*(A2) + .2*(B4)$);
\coordinate (C3) at ($.8*(A3) + .2*(B1)$);

\draw (A2) -- (C2) -- (B2);
\draw (A3) -- (C3) -- (B3);
\draw[red, line width=1]
(C2) -- (C3) (A1) -- (A2) (B1) -- (B2)
(B3) -- (B4) (A3) -- (A4) (A5) -- (B5);
\end{scope}
\end{tikzpicture}
 \caption{$G\! B\! P_5$}
\label{subfig:al1}
\end{subfigure}
\begin{subfigure}[b]{0.225\textwidth}
\centering
﻿\begin{tikzpicture}

\begin{scope}[xshift=2.5cm]
\foreach \a in {0,...,5}
{
\coordinate (A\a) at (72*\a + 90:1);
\coordinate (B\a) at (72*\a+ 90:.5);
}
\foreach \a [evaluate=\a as \b using \a+1] in {0,...,4}
{
\draw (A\a) --(A\b);
\draw (B\a) --(B\b);
}
\foreach \a in {0,1,4}
{
\draw (A\a) -- (B\a);
}
\coordinate (C2) at ($.8*(A2) + .2*(B4)$);
\coordinate (C3) at ($.8*(A3) + .2*(B1)$);

\draw (A2) -- (C2) -- (B2);
\draw (A3) -- (C3) -- (B3);
\draw (C2)  -- (C3);

\draw[red, line width=1]
(C2) -- (A2) (C3) -- (B3) (B1) -- (B2)
(B5) -- (B4) (A3) -- (A4) (A5) -- (A1);

\end{scope}
\end{tikzpicture}
 \caption{$G\! B\! P_5$}
\label{subfig:al1a}
\end{subfigure}
\caption{}
\label{fig:al}
\end{figure}

Some of these polyhedra have their own names: the polyhedra \subref{subfig:al3} is known as \emph{Scutoid} (\cite{GVTFCLVGB:18}); the polyhedra \subref{subfig:al1} or \subref{subfig:al1a}
$G\! B\! P_5$ is the \emph{Gyrobipentaprism}, that is the result of pasting together two pentagonal prism along a lateral face by a $\frac{\pi}{2}$-turn; and the
polyhedra~\subref{subfig:al6} is the \emph{D\"{u}rer-Solid}, the  solid depicted in an engraving entitled \texttt{Melencolia\!~I} by Albrecht D\"{u}rer in 1514.
The graph $\Gamma(P_O)$ is in the Figure~\ref{fig:grafo-oct}.

\begin{figure}[ht]
\centering
%\tikzexternaldisable
\begin{tikzpicture}

\coordinate (A) at (0,0);
\coordinate (B) at (-3,0);
\coordinate (C) at (3,0);
\coordinate (D) at (-2,1.5);
\coordinate (E) at (2,1.5);
\coordinate (F) at (-4,1.5);
\coordinate (G) at (0,3);

\fill (A) node[below] {$[24]$} node[above] {\subref{subfig:al3}} circle [radius=.1];
\fill (B) node[below] {$[4]$} node[left] {\subref{subfig:al6}} circle [radius=.1];
\fill (C) node[below] {$[4]$} node[right] {\subref{subfig:al5}} circle [radius=.1];
\fill (D) node[right=6pt] {$[12]$} node[right=8mm] {\subref{subfig:al2}} circle [radius=.1];
\fill (E) node[right] {$[12]$}  node[right=6mm] {\subref{subfig:al1a}} circle [radius=.1];
\fill (F) node[below] {$[2]$} node[left] {\subref{subfig:al4}}   circle [radius=.1];
\fill (G) node[above] {$[6]$} node[above right] {\subref{subfig:al1}} circle [radius=.1];
\draw[->-=.5] (A) -- node[below] {$(1,-6)$} (B);
\draw[->-=.5] (A) -- node[below] {$(1,-6)$} (C);
\draw[->-=.5] (A) -- node[left=4pt] {$(2,-4)$} (D);
\draw[->-=.5] (A) -- node[right=4pt] {$(2,-4)$} (E);
\draw[->-=.5] (D) -- node[above] {$(1,-6)$} (F);
\draw[->-=.5] (D) -- node[above left] {$(1,-2)$} (G);
\draw[->-=.5] (E) -- node[above right] {$(2,-4)$} (G);
\end{tikzpicture}
 %\tikzexternalenable
\caption{}
\label{fig:grafo-oct}
\end{figure}

In the following sections, we describe the topology of all these smoothings.
With the help of \texttt{Sagemath}~\cite{Sagemath}, we describe some fundamental subgroups.
The code that allowed us to find it can be found in
\url{https://github.com/enriqueartal/SingularQuadricIntersections}; it can be executed
in a local installation of \texttt{Sagemath} or using \texttt{Binder}~\cite{binder}.

\begin{proposition}
The fundamental group of $Z(P_O)$ is free abelian of rank~$4$.
The fundamental group of the manifold $Z_T$ is the derived
of $\check{G}'$ in{\rm~\eqref{eq:group}} and its abelianization is free
of rank~$100$.
\end{proposition}

\begin{figure}[ht]
\begin{subfigure}[b]{0.225\textwidth}
\centering
﻿\begin{tikzpicture}[scale=.4]

\borromeo
\end{tikzpicture}
 \caption*{}
\end{subfigure}
\setcounter{subfigure}{0}
\begin{subfigure}[b]{0.225\textwidth}
\centering
﻿\begin{tikzpicture}[scale=.4]

\borromeoa

\dsa{(R1-2)}{(R1-4)}{(R1-3)}{(R1-5)}
\dsa{(R2-1)}{(R2-4)}{(R2-3)}{(R2-6)}
\dsa{(R3-1)}{(R3-2)}{(R3-5)}{(R3-6)}
\dsa{(R4-1)}{(R4-2)}{(R4-5)}{(R4-6)}
\dsa{(R5-1)}{(R5-3)}{(R5-4)}{(R5-6)}
\dsa{(R6-2)}{(R6-3)}{(R6-4)}{(R6-5)}

\node at ($1/3*(P1)+1/3*(P2)+1/3*(P3)$) {$5$};
\node at ($1/3*(P1)+1/3*(P2)+3/7*(P4)$) {$3$};
\node at ($1/3*(P1)+1/3*(P3)+3/7*(P5)$) {$4$};
\node at ($1/3*(P2)+1/3*(P3)+3/7*(P6)$) {$4$};
\node at ($3/4*(R2-4)+1*(R2-6)$) {$6$};
\node at ($3/4*(R3-5)+1*(R3-6)$) {$5$};
\node at ($3/4*(R1-4)+3/4*(R1-5)$) {$5$};
\node at ($1/2*(R5-4)+2/3*(R5-6)$) {$3$};
\end{tikzpicture}
 \caption{}
\label{subfig:oc2}
\end{subfigure}
\begin{subfigure}[b]{0.225\textwidth}
\centering
﻿\begin{tikzpicture}[scale=.4]
\borromeoa

\dsa{(R1-2)}{(R1-4)}{(R1-3)}{(R1-5)}
\dsa{(R2-1)}{(R2-3)}{(R2-4)}{(R2-6)}
\dsa{(R3-1)}{(R3-2)}{(R3-5)}{(R3-6)}
\dsa{(R4-1)}{(R4-2)}{(R4-5)}{(R4-6)}
\dsa{(R5-1)}{(R5-3)}{(R5-4)}{(R5-6)}
\dsa{(R6-2)}{(R6-3)}{(R6-4)}{(R6-5)}

\node at ($1/3*(P1)+1/3*(P2)+1/3*(P3)$) {$4$};
\node at ($1/3*(P1)+1/3*(P2)+3/7*(P4)$) {$4$};
\node at ($1/3*(P1)+1/3*(P3)+3/7*(P5)$) {$4$};
\node at ($1/3*(P2)+1/3*(P3)+3/7*(P6)$) {$5$};
\node at ($3/4*(R2-4)+1*(R2-6)$) {$5$};
\node at ($3/4*(R3-5)+1*(R3-6)$) {$5$};
\node at ($3/4*(R1-4)+3/4*(R1-5)$) {$6$};
\node at ($1/2*(R5-4)+2/3*(R5-6)$) {$3$};

\end{tikzpicture}
 \caption{}
\label{subfig:oc3}
\end{subfigure}
\begin{subfigure}[b]{0.225\textwidth}
\centering
﻿\begin{tikzpicture}[scale=.4]

\borromeoa
\dsa{(R1-2)}{(R1-4)}{(R1-3)}{(R1-5)}
\dsa{(R2-1)}{(R2-4)}{(R2-3)}{(R2-6)}
\dsa{(R3-1)}{(R3-5)}{(R3-2)}{(R3-6)}
\dsa{(R4-1)}{(R4-2)}{(R4-5)}{(R4-6)}
\dsa{(R5-1)}{(R5-3)}{(R5-4)}{(R5-6)}
\dsa{(R6-2)}{(R6-3)}{(R6-4)}{(R6-5)}

\node at ($1/3*(P1)+1/3*(P2)+1/3*(P3)$) {$6$};
\node at ($1/3*(P1)+1/3*(P2)+3/7*(P4)$) {$3$};
\node at ($1/3*(P1)+1/3*(P3)+3/7*(P5)$) {$3$};
\node at ($1/3*(P2)+1/3*(P3)+3/7*(P6)$) {$3$};
\node at ($3/4*(R2-4)+1*(R2-6)$) {$6$};
\node at ($3/4*(R3-5)+1*(R3-6)$) {$6$};
\node at ($3/4*(R1-4)+3/4*(R1-5)$) {$6$};
\node at ($1/2*(R5-4)+2/3*(R5-6)$) {$3$};
\end{tikzpicture}

 \caption{}
\label{subfig:oc4}
\end{subfigure}

\begin{subfigure}[b]{0.225\textwidth}
\centering
﻿\begin{tikzpicture}[scale=.4]

\borromeoa

\dsa{(R1-2)}{(R1-3)}{(R1-4)}{(R1-5)}
\dsa{(R2-1)}{(R2-3)}{(R2-4)}{(R2-6)}
\dsa{(R3-1)}{(R3-2)}{(R3-5)}{(R3-6)}
\dsa{(R4-1)}{(R4-2)}{(R4-5)}{(R4-6)}
\dsa{(R5-1)}{(R5-3)}{(R5-4)}{(R5-6)}
\dsa{(R6-2)}{(R6-3)}{(R6-4)}{(R6-5)}

\node at ($1/3*(P1)+1/3*(P2)+1/3*(P3)$) {$3$};
\node at ($1/3*(P1)+1/3*(P2)+3/7*(P4)$) {$5$};
\node at ($1/3*(P1)+1/3*(P3)+3/7*(P5)$) {$5$};
\node at ($1/3*(P2)+1/3*(P3)+3/7*(P6)$) {$5$};
\node at ($3/4*(R2-4)+1*(R2-6)$) {$5$};
\node at ($3/4*(R3-5)+1*(R3-6)$) {$5$};
\node at ($3/4*(R1-4)+3/4*(R1-5)$) {$5$};
\node at ($1/2*(R5-4)+2/3*(R5-6)$) {$3$};

\end{tikzpicture}
 \caption{}
\label{subfig:oc6}
\end{subfigure}
\begin{subfigure}[b]{0.225\textwidth}
\centering
﻿\begin{tikzpicture}[scale=.4]

\borromeoa
\dsa{(R1-2)}{(R1-3)}{(R1-4)}{(R1-5)}
\dsa{(R2-1)}{(R2-3)}{(R2-4)}{(R2-6)}
\dsa{(R3-1)}{(R3-5)}{(R3-2)}{(R3-6)}
\dsa{(R4-1)}{(R4-5)}{(R4-2)}{(R4-6)}
\dsa{(R5-1)}{(R5-3)}{(R5-4)}{(R5-6)}
\dsa{(R6-2)}{(R6-3)}{(R6-4)}{(R6-5)}

\node at ($1/3*(P1)+1/3*(P2)+1/3*(P3)$) {$4$};
\node at ($1/3*(P1)+1/3*(P2)+3/7*(P4)$) {$6$};
\node at ($1/3*(P1)+1/3*(P3)+3/7*(P5)$) {$4$};
\node at ($1/3*(P2)+1/3*(P3)+3/7*(P6)$) {$4$};
\node at ($3/4*(R2-4)+1*(R2-6)$) {$4$};
\node at ($3/4*(R3-5)+1*(R3-6)$) {$6$};
\node at ($3/4*(R1-4)+3/4*(R1-5)$) {$4$};
\node at ($1/2*(R5-4)+2/3*(R5-6)$) {$4$};
\end{tikzpicture}

 \caption{}
\label{subfig:oc5}
\end{subfigure}
\begin{subfigure}[b]{0.225\textwidth}
\centering
﻿\begin{tikzpicture}[scale=.4]
\borromeoa
\dsa{(R1-2)}{(R1-4)}{(R1-3)}{(R1-5)}
\dsa{(R2-1)}{(R2-4)}{(R2-3)}{(R2-6)}
\dsa{(R3-1)}{(R3-2)}{(R3-5)}{(R3-6)}
\dsa{(R4-1)}{(R4-5)}{(R4-2)}{(R4-6)}
\dsa{(R5-1)}{(R5-3)}{(R5-4)}{(R5-6)}
\dsa{(R6-2)}{(R6-3)}{(R6-4)}{(R6-5)}

\node at ($1/3*(P1)+1/3*(P2)+1/3*(P3)$) {$5$};
\node at ($1/3*(P1)+1/3*(P2)+3/7*(P4)$) {$4$};
\node at ($1/3*(P1)+1/3*(P3)+3/7*(P5)$) {$4$};
\node at ($1/3*(P2)+1/3*(P3)+3/7*(P6)$) {$4$};
\node at ($3/4*(R2-4)+1*(R2-6)$) {$5$};
\node at ($3/4*(R3-5)+1*(R3-6)$) {$5$};
\node at ($3/4*(R1-4)+3/4*(R1-5)$) {$5$};
\node at ($1/2*(R5-4)+2/3*(R5-6)$) {$4$};
\end{tikzpicture}
 \caption{$G\negthinspace B\negthinspace P_5$}
\label{subfig:oc1}
\end{subfigure}
\begin{subfigure}[b]{0.225\textwidth}
\centering
﻿\begin{tikzpicture}[scale=.4]
\borromeoa
\dsa{(R1-2)}{(R1-4)}{(R1-3)}{(R1-5)}
\dsa{(R2-1)}{(R2-3)}{(R2-4)}{(R2-6)}
\dsa{(R3-1)}{(R3-2)}{(R3-5)}{(R3-6)}
\dsa{(R4-1)}{(R4-5)}{(R4-2)}{(R4-6)}
\dsa{(R5-1)}{(R5-3)}{(R5-4)}{(R5-6)}
\dsa{(R6-2)}{(R6-3)}{(R6-4)}{(R6-5)}

\node at ($1/3*(P1)+1/3*(P2)+1/3*(P3)$) {$4$};
\node at ($1/3*(P1)+1/3*(P2)+3/7*(P4)$) {$5$};
\node at ($1/3*(P1)+1/3*(P3)+3/7*(P5)$) {$4$};
\node at ($1/3*(P2)+1/3*(P3)+3/7*(P6)$) {$5$};
\node at ($3/4*(R2-4)+1*(R2-6)$) {$4$};
\node at ($3/4*(R3-5)+1*(R3-6)$) {$5$};
\node at ($3/4*(R1-4)+3/4*(R1-5)$) {$5$};
\node at ($1/2*(R5-4)+2/3*(R5-6)$) {$4$};
\end{tikzpicture}
 \caption{$G\negthinspace B\negthinspace P_5$}
\label{subfig:oc1a}
\end{subfigure}
\caption{Stereographic presentation}
\label{fig:oc}
\end{figure}
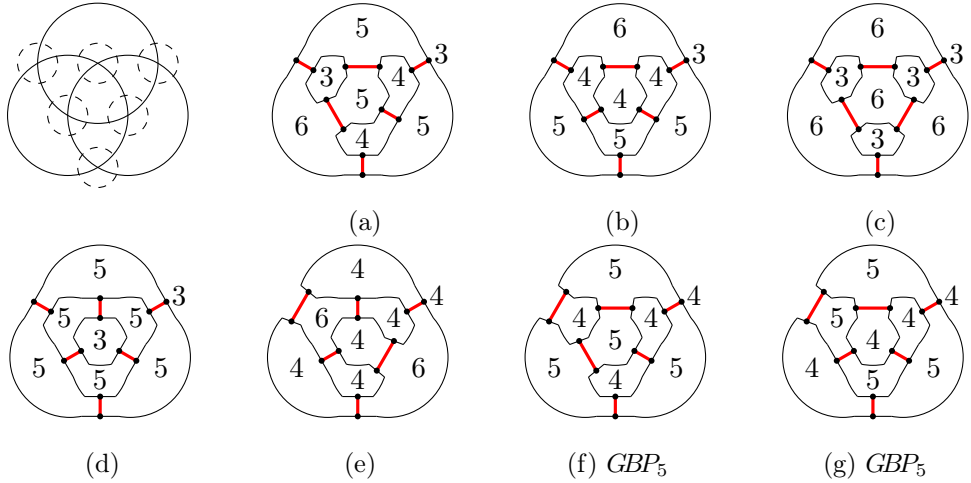

\subsection{Topology of the manifolds obtained in cases  \texorpdfstring{\subref{subfig:al2}-\subref{subfig:al5}}{(a)-(e)}}\label{subsec:especial}
\mbox{}

The topology can be described in most cases using the fact that the resulting polyhedra are truncations of simpler polyhedra. These polyhedra are connected
by flips of type~$(2,2)$. We are going to relate them
with the distinct smoothings of the octahedron.

\begin{figure}[ht]
\centering
\begin{subfigure}[b]{0.3\textwidth}
\centering
\begin{tikzpicture}
\foreach \a in {0,...,4}
{
\coordinate (A\a) at (90*\a + 45:1);
\coordinate (B\a) at (90*\a + 45:2/5);
}
\foreach \a[evaluate=\a as \b using \a+1] in {0,...,3}
{
\draw (A\a) -- (A\b);
\draw (B\a) -- (B\b);
\draw (A\a) -- (B\a);
}
\draw[line width=1.2] ($.8*(A2) + .2*(A1)$) --
($.8*(A2) + .2*(A3)$) --
($1/3*(A2)+2/3*(B2)$) -- cycle;

\draw[line width=1.2] ($.8*(A3) + .2*(A0)$) --
($.8*(A3) + .2*(A2)$) --
($1/3*(A3)+2/3*(B3)$) -- cycle;
\end{tikzpicture}
 \caption{}
\label{subfig:tr2}
\end{subfigure}
\begin{subfigure}[b]{0.3\textwidth}
\centering
﻿\begin{tikzpicture}[scale=1]
\foreach \a in {0,...,5}
{
\coordinate (A\a) at (72*\a + 90:1);
\coordinate (B\a) at (72*\a+ 90:.5);
}
\foreach \a [evaluate=\a as \b using \a+1] in {0,...,4}
{
\draw (A\a) --(A\b);
\draw (B\a) --(B\b);
}
\foreach \a in {0,...,4}
{
\draw (A\a) -- (B\a);
}
\draw[line width=1.2]
($.5*(A0) + .5*(B0)$) --
($.5*(B0) + .5*(B1)$) --
($.5*(B0) + .5*(B4)$) -- cycle;
\end{tikzpicture}
 \caption{}
\label{subfig:tr3}
\end{subfigure}
\begin{subfigure}[b]{0.3\textwidth}
\centering
\begin{tikzpicture}
\coordinate (O) at (0,0);
\foreach \a in {0,1,2}
{
\coordinate (A\a) at (120*\a+90:1);
\draw (O) -- (A\a);
}
\draw (A0) -- (A1) -- (A2) -- cycle;
\coordinate (A10) at ($1/4*(A0) + 3/4*(A1)$);
\coordinate (A12) at ($1/4*(A2) + 3/4*(A1)$);
\coordinate (A1O) at ($.5*(O) + .5*(A1)$);
\coordinate (A20) at ($1/4*(A0) + 3/4*(A2)$);
\coordinate (A21) at ($1/4*(A1) + 3/4*(A2)$);
\coordinate (A2O) at ($.5*(O) + .5*(A2)$);
\coordinate (A02) at ($3/4*(A0) + 1/4*(A2)$);
\coordinate (A01) at ($1/4*(A1) + 3/4*(A0)$);
\coordinate (A0O) at ($.5*(O) + .5*(A0)$);
\coordinate (O0) at ($2/3*(O) + 1/3*(A0)$);
\coordinate (O1) at ($2/3*(O) + 1/3*(A1)$);
\coordinate (O2) at ($2/3*(O) + 1/3*(A2)$);

\draw[line width=1.2] (A10) -- (A12) -- (A1O) --cycle;
\draw[line width=1.2] (A20) -- (A21) -- (A2O) --cycle;
\draw[line width=1.2] (A02) -- (A01) -- (A0O) --cycle;
\draw[line width=1.2] (O0) -- (O1) -- (O2) -- cycle;

\end{tikzpicture}
 \caption{}
\label{subfig:tr4}
\end{subfigure}

\begin{subfigure}[b]{0.3\textwidth}
\centering
\begin{tikzpicture}
\foreach \a in {0,...,4}
{
\coordinate (A\a) at (90*\a + 45:1);
\coordinate (B\a) at (90*\a + 45:2/5);
}
\foreach \a[evaluate=\a as \b using \a+1] in {0,...,3}
{
\draw (A\a) -- (A\b);
\draw (B\a) -- (B\b);
\draw (A\a) -- (B\a);
}
\draw[line width=1.2] ($.5*(B1) + .5*(B2)$) --
($.5*(B1) + .5*(B0)$) --
($.5*(B1)+.5*(A1)$) -- cycle;

\draw[line width=1.2] ($.8*(A3) + .2*(A0)$) --
($.8*(A3) + .2*(A2)$) --
($1/3*(A3)+2/3*(B3)$) -- cycle;
\end{tikzpicture}
 \caption{}
\label{subfig:tr6}
\end{subfigure}
\begin{subfigure}[b]{0.3\textwidth}
\centering
﻿\begin{tikzpicture}[scale=1]
\foreach \a in {0,...,5}
{
\coordinate (A\a) at (72*\a + 90:1);
\coordinate (B\a) at (72*\a+ 90:.5);
}
\foreach \a [evaluate=\a as \b using \a+1] in {0,...,4}
{
\draw (A\a) --(A\b);
\draw (B\a) --(B\b);
}
\foreach \a in {0,...,4}
{
\draw (A\a) -- (B\a);
}
\draw[line width=1.2]
($.5*(B2) + .5*(B3)$) --
($.3*(A2) + .7*(A3)$) --
($.3*(A4) + .7*(A3)$) --
($.5*(B3) + .5*(B4)$) -- cycle;
\end{tikzpicture}
 \caption{}
\label{subfig:tr5}
\end{subfigure}
\begin{subfigure}[b]{0.3\textwidth}
\centering
﻿\begin{tikzpicture}[scale=1]
\foreach \a in {0,...,5}
{
\coordinate (A\a) at (72*\a + 90:1);
\coordinate (B\a) at (72*\a+ 90:.5);
}
\foreach \a [evaluate=\a as \b using \a+1] in {0,...,4}
{
\draw (A\a) --(A\b);
\draw (B\a) --(B\b);
}
\foreach \a in {0,...,4}
{
\draw (A\a) -- (B\a);
}
\draw[line width=1.2]
($.6*(B2) + .4*(B1)$) --
($.4*(A2) + .6*(B2)$) --
($.4*(A3) + .6*(B3)$) --
($.6*(B3) + .4*(B4)$) -- cycle;

\end{tikzpicture}
 \caption{$G\negthinspace B\negthinspace P_5$}
\label{subfig:tr1}
\end{subfigure}
\caption{Some truncations of simple polyhedra }
\label{f1a}
\end{figure}
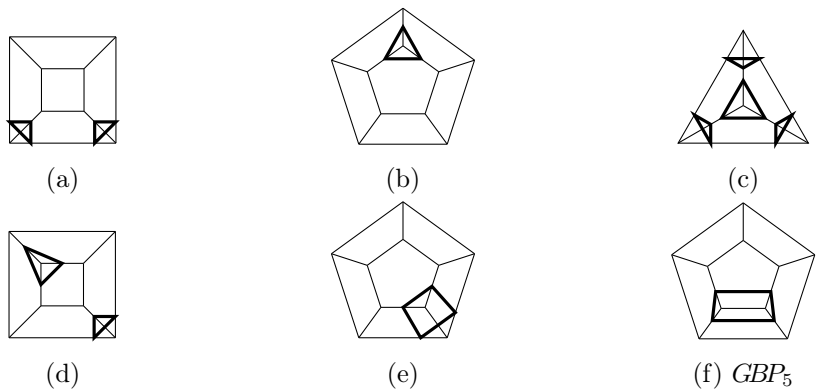

The cases \subref{subfig:al2}, \subref{subfig:al3},
\subref{subfig:al4}, and \subref{subfig:al6}
are obtained
from well-known polyhedra by truncating vertices,
while \subref{subfig:al1}, \subref{subfig:al1a}, and \subref{subfig:al5}
are the \emph{bevellings} \footnote{In spanish, the term \emph{biselados}
was suggested by Alberto Verjovsky.} of the pentagonal prism by cutting
along a horizontal, for \subref{subfig:al1}, \subref{subfig:al1a}, or a vertical edge,
for \subref{subfig:al5}.

They can be also described as the double bevellings of the cube; the common case of
\subref{subfig:al1} and \subref{subfig:al1a} is the bevelling on two non-parallel disjoint edges
and is the prototype of a polytope whose \emph{moment-angle manifold}
(the intersection of ellipsoids obtained by duplicating every $A_i$ and
whose quotient by a torus gives the same polytope) has non-trivial
Massey products, and it is plausible that the same is true for $Z(GBP)$,
see \cite[p.~26]{Pa}, where the Figure~\ref{fig:bibi} appears.
In particular, it is not a connected sum of elementary manifolds, as the rest of them.
The topology of this intersection will be described in the next section.

\begin{figure}[ht]
\begin{center}
\begin{tikzpicture}
\coordinate (A1) at (-1,-1.5);
\coordinate (A2) at (1,-1.5);
\coordinate (A3) at (-1,1.5);
\coordinate (A4) at (.5,1.5);
\coordinate (A5) at (1,1);

\coordinate (T1) at (-.75, .25);
\coordinate (B1) at ($(A1)+(T1)$);
\coordinate (B2) at ($(A3)+(T1)$);

\coordinate (T2) at (1.5, 1);
\coordinate (C1) at ($(A2)+(T2)$);
\coordinate (C2) at ($(A5)+(T2)$);
\coordinate (C3) at ($(A4)+(T2)$);

\coordinate (T3) at (.5, 1);
\coordinate (D1) at ($(A1)+(T3)$);
\coordinate (D2) at ($(A3)+(T3)$);

\draw (A2) -- (A1) -- (A3) -- (A4) -- (A5) -- cycle;
\draw (A1) -- (B1) -- (B2) -- (A3);

\draw (A2) -- (C1) -- (C2) -- (A5);
\draw (A4) -- (C3) -- (C2);

\draw (B2) -- (D2) --(C3);

\draw[dashed] (B1) -- (D1) -- (D2) (D1) -- (C1);

\node at ($.65*(A3) + .35*(C3)$) {$v_1$};
\node at ($.5*(A1) + .5*(B2)$) {$w_1$};
\draw[dashed,->] ($.5*(A1) + .5*(A2) - (0,.25)$) node[below] {$v_2$} -- ($.5*(A1) + .5*(A2) + (0,.25)$);
\node at ($.5*(A4) + .5*(C2)$) {$w_2$};
\node at ($.5*(A2) + .5*(C2)$) {$v_3$};
\draw[dashed,->] ($.5*(B2) + .5*(D2) + (-.25,.25)$) node[left] {$v_4$} -- ($.5*(B2) + .5*(D2) - (-.25,.25)$);
\node at ($.5*(A2) + .5*(A3)$) {$v_5$};
\draw[dashed,->] ($.5*(C3) + .5*(D2) + (0,.25)$) node[above] {$v_6$} -- ($.5*(C3) + .5*(D2) - (0,.25)$);

\end{tikzpicture}
 \caption{}
\label{fig:bibi}
\end{center}
\end{figure}

We now describe all the other cases.
Let us denote by $\mathcal{S}_{g,n}$ the closed orientable surface of genus $g$
with $n$ boundary components (for short, $\mathcal{S}_{g}=\mathcal{S}_{g,0}$).
Since \subref{subfig:al5}
is an hexagonal prism, then the corresponding intersection is diffeomorphic to
$\mathcal{S}_{17} \times\mathbb{S}^1$.

For the remaining cases we use Proposition~\ref{prop:truncar} for truncations of
simple polytopes.

Note that \subref{subfig:al2} and \subref{subfig:al6} are different double truncations of a cube (there are $3$ possible ones). Hence the two intersections are diffeomorphic to
$4 (\mathbb{S}^1 \times \mathbb{S}^1 \times \mathbb{S}^1)\# 29 (\mathbb{S}^{2} \times \mathbb{S}^1)$.
The case \subref{subfig:al3} is a truncation of the pentagonal prism, so the intersection is diffeomorphic to
$ 2 (\mathcal{S}_5 \times \mathbb{S}^1)\# 15 (\mathbb{S}^{2} \times \mathbb{S}^1)$.
Finally, \subref{subfig:al4}
is obtained by truncating the four vertices of the tetrahedron so its intersection is diffeomorphic to
$49 (\mathbb{S}^{2} \times \mathbb{S}^1)$.

\subsection{The manifold \texorpdfstring{$Z$}{Z} associated to the Gyrobipentaprism}
\mbox{}

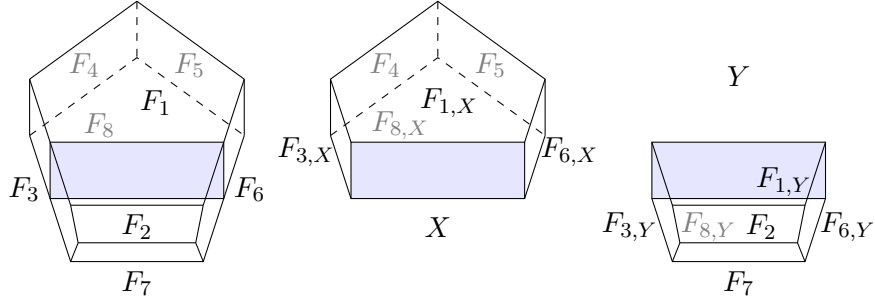
\begin{figure}[ht]
\centering
\begin{tikzpicture}
\foreach \a in {0,...,5}
{
\coordinate (A\a) at (72*\a + 90:1.5);
\coordinate (B\a) at ($(A\a)-.75*(0,1)$);
}
\filldraw[fill=blue!10!white] ($.5*(A1)+.5*(A2)$)
-- ($.5*(B1)+.5*(B2)$) --
($.5*(B3)+.5*(B4)$) --
($.5*(A3)+.5*(A4)$) -- cycle;
\foreach \a [evaluate=\a as \b using \a+1] in {0,...,4}
{
\draw (A\a) --(A\b);
}
\coordinate (C2) at ($.6*(A2)+.4*(B2)+(.1,-.2)$);
\coordinate (C3) at ($.6*(A3)+.4*(B3)+(-.1,-.2)$);
\draw (B1) -- (B2) -- (B3) -- (B4);
\draw[dashed] (B1) -- (B0) -- (B4);
\draw[dashed] (A0) -- (B0);
\draw (A1) -- (B1);
\draw (A4) -- (B4);
\draw (A2) -- (C2) -- (B2);
\draw (A3) -- (C3) -- (B3);
\draw (C2) -- (C3);

\node at (.25,.15) {$F_1$};
\node at (0,-1.5) {$F_2$};
\node at (-1.5,-1) {$F_{3}$};
\node[gray] at (-.7,.65) {$F_{4}$};
\node[gray] at (.7,.65) {$F_{5}$};
\node at (1.5,-1) {$F_{6}$};
\node at (0,-2.25) {$F_7$};
\node[gray] at (-.5,-.15) {$F_8$};

\begin{scope}[xshift=4cm]
\foreach \a in {0,...,5}
{
\coordinate (A\a) at (72*\a + 90:1.5);
\coordinate (B\a) at ($(A\a)-.75*(0,1)$);
}
\filldraw[fill=blue!10!white] ($.5*(A1)+.5*(A2)$)
-- ($.5*(B1)+.5*(B2)$) --
($.5*(B3)+.5*(B4)$) --
($.5*(A3)+.5*(A4)$) -- cycle;
\draw ($.5*(A3)+.5*(A4)$) -- (A4) --
(A0) -- (A1) -- ($.5*(A1)+.5*(A2)$);

\draw ($.5*(B3)+.5*(B4)$) -- (B4);
\draw ($.5*(B1)+.5*(B2)$) -- (B1);
\draw[dashed] (B1) -- (B0) -- (B4);
\draw[dashed] (A0) -- (B0);
\draw (A1) -- (B1);
\draw (A4) -- (B4);
\node at (-90:1.5) {$X$};

\node at (.15,.15) {$F_{1,X}$};
\node at (-1.75,-.5) {$F_{3,X}$};
\node[gray] at (-.7,.65) {$F_{4}$};
\node[gray] at (.7,.65) {$F_{5}$};
\node at (1.75,-.5) {$F_{6,X}$};
\node[gray] at (-.5,-.15) {$F_{8,X}$};

\end{scope}

\begin{scope}[xshift=8cm]
\foreach \a in {0,...,5}
{
\coordinate (A\a) at (72*\a + 90:1.5);
\coordinate (B\a) at ($(A\a)-.75*(0,1)$);
}
\filldraw[fill=blue!10!white] ($.5*(A1)+.5*(A2)$)
-- ($.5*(B1)+.5*(B2)$) --
($.5*(B3)+.5*(B4)$) --
($.5*(A3)+.5*(A4)$) -- cycle;
\draw ($.5*(A1)+.5*(A2)$) -- (A2) -- (A3)
-- ($.5*(A3)+.5*(A4)$);
\draw ($.5*(B1)+.5*(B2)$) -- (B2) -- (B3)
-- ($.5*(B3)+.5*(B4)$);

\coordinate (C2) at ($.6*(A2)+.4*(B2)+(.1,-.2)$);
\coordinate (C3) at ($.6*(A3)+.4*(B3)+(-.1,-.2)$);
\draw (A2) -- (C2) -- (B2);
\draw (A3) -- (C3) -- (B3);
\draw (C2) -- (C3);

\node at (90:.5) {$Y$};

\node at (.6,-.9) {$F_{1,Y}$};
\node at (.3,-1.5) {$F_2$};
\node at (-1.45,-1.5) {$F_{3,Y}$};
\node at (1.45,-1.5) {$F_{6,Y}$};
\node at (0,-2.25) {$F_7$};
\node[gray] at (-.4,-1.5) {$F_{8,Y}$};

\end{scope}
\end{tikzpicture}
 \caption{The polytope $G\negthinspace B\negthinspace P_5$ in \ref{subfig:al1} as union of $X$ and $Y$.}
\label{polyXY}
\end{figure}

We devote this section to the manifold $Z$ associated to polyhedra Gyrobipentaprism  $G\negthinspace B\negthinspace P_5$ in
Figure~\ref{subfig:al1}, which has not a straightforward description as for the other cases.
Following the same ideas of the truncation we cut the polytope $G\negthinspace B\negthinspace P_5$
along a rectangle as in Figure~\ref{polyXY} obtaining two new  polytopes $X$ and $Y$.
We denote the faces of $G\negthinspace B\negthinspace P_5,X,Y$ as in this figure. Both $X$ and $Y$ are pentagonal prisms.
In order to construct $Z$, we must perform reflections on all the faces except the ones coming from the
cutting rectangle. We obtain two $3$-manifolds with four components in the boundary, all of them tori.
The manifold $Z$ is obtained by suitably gluing four copies of both manifolds along the boundary components.
The precise description of the gluing and the final result occupies the rest of this section.

Let us denote the faces in Figure~\ref{polyXY}.
The faces $F_1, F_{1,X}, F_{1,Y}$ (resp. $F_1, F_{1,X}, F_{1,Y}$) are the upper (resp. lower) bases. The face $F_2$ (resp. $F_7$) is the upper
(resp. lower)
face of the bevelling. The faces $F_3, F_{3,X}, F_{3,Y}3$ are in the front left-hand side.
The faces $F_4,F_5$ are the backward lateral faces.

Since both polytopes are \emph{equal}, we describe them together. Consider a pentagonal prism where we mirror
along all faces but one of the lateral ones. We consider several steps:
\begin{enumerate}[label=(T\arabic{enumi})]
\item Reflection along the four lateral faces that are mirrors. The result is the product
of the reflection of a pentagon in four sides by the interval $I$. We can see the result in
Figure~\ref{pentagonoreflejado4}, i.e., a torus with four holes by interval: $\mathcal{S}_{1,4} \times I$.

\begin{figure}[ht]
\begin{center}
\begin{tikzpicture}[scale=2]

\bloqueiv{(0,0)}
\bloqueiv{(1,0)}
\bloqueiv{(0,-1)}
\bloqueiv{(1,-1)}

\draw[->-=.5] (-1,-1) -- (-1, 1);
\draw[->-=.5] (1,-1) -- (1, 1);
\draw[->-=.475, ->-=.525] (-1,-1) -- (1, -1);
\draw[->-=.475, ->-=.525] (-1,1) -- (1, 1);

\node at (1.5, 0) {$\times\mathbb{S}^1$};
\end{tikzpicture}
 \caption{The manifold associated to a pentagonal prism with a non-mirroring lateral face.}\label{pentagonoreflejado4}
\end{center}
\end{figure}
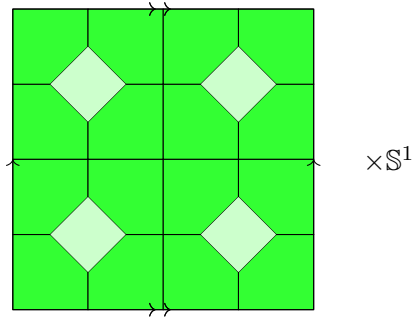

\item Reflecting on the basis faces. We obtain $\mathcal{S}_{1,4} \times \mathbb{S}^1$, a trivial
$\mathbb{S}^1$-fibre bundle.

\item The boundary of this manifold consists of $4$~tori corresponding to the non-mirroring lateral face of the original pentagonal prism.

\item With this process we denote by $X_e$
the manifold obtained by reflection of $X$ in the lateral faces $F_{3,X}$, $F_4$, $F_5$, $F_{6,X}$, and in the bases $F_{1,X}$ and $F_{8,X}$.

\item Analogously, let $Y_e$
the manifold obtained by reflection of $Y$ in the lateral faces $F_{1,Y}$, $F_2$, $F_7$, $F_{8,Y}$, and in the bases $F_{3,Y}$ and $F_{6,Y}$.

\item The result of the two previous statements can be seen in Figure~\ref{pentagonoreflejadoXY}.

\item The manifold $Z_X$ (associated to $X$ with one missing reflection) is obtained by reflecting
along the faces $F_2, F_7$ which are the faces of $Y$ disjoint to $X$. We obtain four diffeomorphic
connected components labelled by the reflections, $Z_X= X_e \cup X_2 \cup X_7 \cup X_{27}$.

\item Analogously, the manifold $Z_Y$ is obtained by reflecting
along the faces $F_4, F_5$. We obtain four diffeomorphic
connected components labelled by the reflections, $Z_Y= Y_e \cup Y_4 \cup Y_5 \cup Y_{45}$.

\item Let us describe the boundaries of these manifolds using the labellings of the reflections:
\begin{align*}
\partial X_e =& T_e \cup T_4 \cup T_5 \cup T_{45} \\
  \partial X_2 =& T_2 \cup T_{24} \cup T_{25} \cup T_{245} \\
  \partial X_7 =& T_7 \cup T_{47} \cup T_{57} \cup T_{457} \\
  \partial X_{27} =& T_{27} \cup T_{247} \cup T_{257} \cup T_{2457} \\
  \partial Y_e =& T_e \cup T_2 \cup T_7 \cup T_{27} \\
  \partial Y_4 =& T_4 \cup T_{24} \cup T_{47} \cup T_{247} \\
  \partial Y_5 =& T_5 \cup T_{25} \cup T_{57} \cup T_{257} \\
 \partial Y_{45} =& T_{45} \cup T_{245} \cup T_{457} \cup T_{2457}
\end{align*}

\item Each manifold is a trivial fibration. In order to get $Z$ we must glue them
along the boundary components, exchanging sections and fibers. This is how
Waldhausen graph manifolds are constructed using Neumann's plumbing, see~\cite{Waldhausen1967, Neumann1981}
\end{enumerate}

Let us describe the manifold~$Z$. Consider the complete bipartite graph $K_{4,4}$. We associate to each
vertex an oriented $\mathbb{S}^1$-bundle with base a $2$-torus and Euler number~$0$, i.e.,
a manifold diffeomorphic to $\mathbb{T}^3$. On each fiber bundle we take out an open fibered
solid torus for each edge. Then, we glue along the tori associated to each edge interchanging
fiber and section.

This description determines completely~$Z$ and it is possible to give a presentation
of its fundamental group, see e.g.~\cite{ACM:20}. For example the homology
is free of rank~$31$.

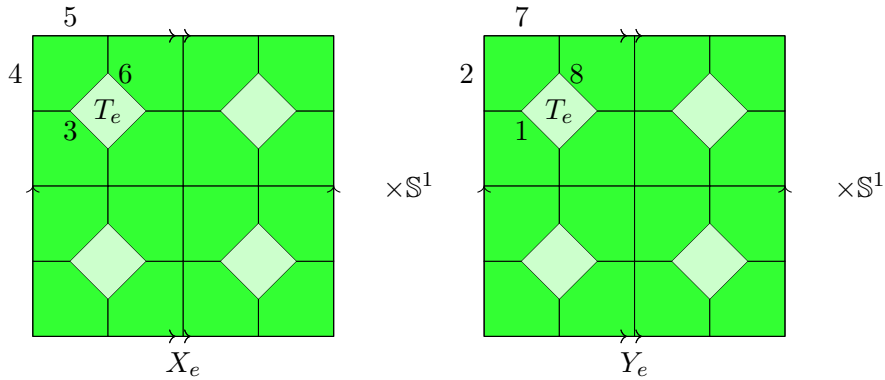
\begin{figure}[ht]
\begin{center}
\begin{tikzpicture}[scale=2]

\bloqueiv{(0,0)}
\bloqueiv{(1,0)}
\bloqueiv{(0,-1)}
\bloqueiv{(1,-1)}

\draw[->-=.5] (-1,-1) -- (-1, 1);
\draw[->-=.5] (1,-1) -- (1, 1);
\draw[->-=.475, ->-=.525] (-1,-1) -- (1, -1);
\draw[->-=.475, ->-=.525] (-1,1) -- (1, 1);

\node at (1.5, 0) {$\times\mathbb{S}^1$};
\node[below=2pt] at (0,-1) {$X_e$};
\node at (-1/2,1/2) {$T_e$};
\node[left] at (-1,3/4) {$4$};
\node[above] at (-3/4,1) {$5$};
\node[right] at (-1/2,3/4) {$6$};
\node[below] at (-3/4,1/2) {$3$};

\begin{scope}[xshift=3cm]
\bloqueiv{(0,0)}
\bloqueiv{(1,0)}
\bloqueiv{(0,-1)}
\bloqueiv{(1,-1)}

\draw[->-=.5] (-1,-1) -- (-1, 1);
\draw[->-=.5] (1,-1) -- (1, 1);
\draw[->-=.475, ->-=.525] (-1,-1) -- (1, -1);
\draw[->-=.475, ->-=.525] (-1,1) -- (1, 1);

\node at (1.5, 0) {$\times\mathbb{S}^1$};
\node[below=2pt] at (0,-1) {$Y_e$};
\node at (-1/2,1/2) {$T_e$};
\node[left] at (-1,3/4) {$2$};
\node[above] at (-3/4,1) {$7$};
\node[right] at (-1/2,3/4) {$8$};
\node[below] at (-3/4,1/2) {$1$};
\end{scope}
\end{tikzpicture}
 \caption{The manifolds $X_e$ and $Y_e$.}\label{pentagonoreflejadoXY}
\end{center}
\end{figure}

There is an alternative description of this manifold which connects it to complex geometry. Let us
consider an elliptic curve $E$ (a compact complex manifold of genus~$1$), and take four points
$p_1,\dots,p_4\in E$. In the projective surface $W:=E\times E$ consider the curve
\[
D:=\bigcup_{i=1}^4 E\times\{p_i\}
\cup
\bigcup_{i=1}^4 \{p_i\}\times E-
\]
Let us take regular neighbourhood of $D$ in $E$ constructed as a union of tubular
neighbourhoods of the irreducible components of~$D$, which at the double
points looks like a polydisk in $\mathbb{C}^2$. Then, $Z$ is homeomorphic
to the boundary of this neighbourhood.

Can we say something about the link induced by the red edges in
Figures~\ref{subfig:al1}
and~\subref{subfig:al1a}?
 
\section{More hyperbolic relations}\label{sec:hyperbolic}

\subsection{Hyperbolic relations between \texorpdfstring{$\check{Z}(BP_3)$}{Z(BP3)} and \texorpdfstring{$\check{Z}(P_O)$}{Z(PO)}}
\mbox{}

There is a hyperbolic structure $T_H$ in the tetrahedron
$ABCD$ in Figure~\ref{fig:hbp} with three ideal points $B$, $C$ and $D$. In this hyperbolic structure the dihedral angle at the edges $AB$, $AC$, and $AD$ is $\frac{\pi}{2}$,  but at  $BC$, $CD$, and $BD$ the dihedral angle is $\frac{\pi}{4}$ . This results follows from the hyperbolic structure given to the triangular bipyramid in Proposition~\ref{prop:hbp}.

The tetrahedron $T_H$ has a orbifold hyperbolic structure with mirror faces and angles $\frac{\pi}{n}$, $n=2,4$. Therefore it defines a tessellation $\mathfrak{T}_T$ in the hyperbolic space.

\begin{thm}
The tessellations $\mathfrak{T}_{BP}$ and $\mathfrak{T}_{P_H}$ are subtessellations in $\mathfrak{T}_T$. The volume of $\check{Z}(P_O)$ is a multiple of the volume of $\check{Z}(BP_3)$.
\[
  \vol (\check{Z}(P_O))= 2^4 \vol (\check{Z}(BP_3).
\]
\end{thm}

\begin{proof}
The bipyramid $BP_H$ (Figure~\ref{fig:hbp}) can be seen as the union of  two tiles, the tetrahedron $ABCD$  and  its reflection on the face $BCD$. The octahedron $P_H$ is the union of eight tiles, the tetrahedron $ABCD$  and  its reflection on the faces $ABC$, $ACD$ and $ABD$. Then
\[
\begin{matrix}
\vol (\check{Z}(BP_3))&=& 2^6 \vol(BP_H)&=& 2^{6+1} \vol(T_H) \\
  \vol(\check{Z}(P_O)) &=& 2^8 \vol(P_O)&=& 2^{8+3} \vol(T_H)
\end{matrix}
\]
implies
\[
  \vol (\check{Z}(P_O))= 2^4 \vol (\check{Z}(BP_3).
  \qedhere
\]
\end{proof}

\subsection{The rhombic dodecahedron}
\mbox{}

There is another interesting hyperbolic polyhedron whose associated tessellation  is a subtessellations in $\mathfrak{T}_T$. Consider the following  hyperbolic rhombic dodecahedron $RD_H$ in the Klein model of the hyperbolic 3-space $\mathbb{H}^3$: The six $4$-vertices are the ideal points $(\pm 1,0,0)$, $(0,\pm1,0)$ and $(0,0,\pm1)$ (the vertices of the octahedron), the eight simple vertices are  the points $(\pm \frac{1}{2}, \pm \frac{1}{2}, \pm \frac{1}{2})$. It is obtained by reflection of the triangular bipyramid $ABCDE$ at its three facets $ABC$, $ACD$ and $ABD$  (Figure~\ref{fig:hbp}), that is, reflection at the three coordinate planes. The edges of $RB_H$ are $BE$, $CE$, $DE$ and their images by the reflections. The dihedral angle at each edge $BE$ is the same as the angle in the triangular bipyramid, that is   $\frac{\pi}{2}$ as was proved in Proposition~\ref{prop:hbp}, and, by symmetry, all the others dihedral angles are also $\frac{\pi}{2}$.
Observe that the group of symmetries  of $RD_H$ is the same that the one of the octahedron.

The tessellation $\mathfrak{T}_{RD}$ in $\mathbb{H}^3$ defined by $RD_H$ is a subtessellations of $\mathfrak{T}_T$. The polyhedron $RD_H$ is made up of 8 triangular bipyramid $BP_H$, 8 tiles of $\mathfrak{T}_{BP}$, then by 16 tiles of $\mathfrak{T}_T$.
\[
\vol (RD_H)= 8 \vol (BP_H) = 16\vol (T_H)
\]

The group $G_{RD}$, subgroup of $\aut\mathbb{H}^3$ generated by the hyperbolic reflections on the twelve  planes containing the faces  of $RD_H$, is the automorphism group of $\mathfrak{T}_{BP}$.
The quotient $\mathbb{H}^3/G_{RD}$ defines a hyperbolic orbifold structure $\mathbf{RD}$ in $RD_H$.

There is a geometric embedding of the rhombic dodecahedron $P_{RD}$  in $\mathbb{R}^{12}$.
This is the polytope associated to an intersection of ellipsoids with $6$ generic isolated singularities.  The manifold $Z(RD)$ has $6\times 2^8$ singular points. Recall that $\check{P}_{RD}$, complement of the non-simple vertices in $P_{RD}$, has an orbifold structure. This orbifold structure is isomorphic to the orbifold structure $\mathbf{RD}$. Then $\check{Z}(RD)$ has a complete hyperbolic structure with $6\times 2^8$ ends of torus type and  $G_{RD}\cong\pi_1^{\text{\rm orb}}(\check{P}_{RD})$.

\[
\vol (\check{Z}(RD))= 2^{12} \vol (RD_H) = 2^{15}\vol (BP_H)=2^{16}\vol (T_H)
\]
\[
\vol (\check{Z}(RD))=2^9\vol (\check{Z}(BP_3)=2^5\vol (\check{Z}(P_O)
\]

The smoothings of the manifold $Z(RD)$ can be studied. In fact the smoothings of the rhombic dodecahedron, having the same symmetry group that the octahedra, reduces to seven topological types. One of them is the hyperbolic dodecahedron studied in \cite{ALdML:16}.

% 
%  
% \bibliographystyle{amsalpha}
% 
% \bibliography{biblio}

\begin{thebibliography}{GVT{\etalchar{+}}18}

\bibitem[ACM20]{ACM:20}
E.~Artal, J.I Cogolludo, and D.~Matei, \emph{Characteristic varieties of graph
  manifolds and quasi-projectivity of fundamental groups of algebraic links},
  Eur. J. Math. \textbf{6} (2020), no.~3, 624--645.

\bibitem[ALL16]{ALdML:16}
E.~Artal, S.~{López de Medrano}, and M.T. Lozano, \emph{The dodecahedron: from
  intersections of quadrics to {B}orromean rings}, A mathematical tribute to
  {P}rofessor {J}.{M}. {M}ontesinos {A}milibia, Dep. Geom. Topol. Fac. Cien.
  Mat. UCM, Madrid, 2016, pp.~85--103.

\bibitem[And70]{and1970b}
E.M. Andreev, \emph{Convex polyhedra of finite volume in {L}oba\v{c}evski\u{\i}
  space}, Mat. Sb. (N.S.) \textbf{83(125)} (1970), 256--260, translated in
  Math. USSR Sb. {\bf 12}, no. 2, 255--259 (1970).

\bibitem[AVS93]{alekseevskij-vinberg-sl:93}
D.V. Alekseevskij, \`E.B. Vinberg, and A.S. Solodovnikov, \emph{Geometry of
  spaces of constant curvature}, Geometry, {II}, Encyclopaedia Math. Sci.,
  vol.~29, Springer, Berlin, 1993, pp.~1--138.

\bibitem[BM06]{B-M}
F.~Bosio and L.~Meersseman, \emph{Real quadrics in {$\mathbb{C}^n$}, complex
  manifolds and convex polytopes}, Acta Math. \textbf{197} (2006), no.~1,
  53--127.

\bibitem[Cha86]{chaperon}
M.~Chaperon, \emph{G\'{e}om\'{e}trie diff\'{e}rentielle et singularit\'{e}s de
  syst\`emes dynamiques}, Ast\'{e}risque (1986), no.~138-139, 440.

\bibitem[DJ91]{DJ1991}
M.W. Davis and T.~Januszkiewicz, \emph{Convex polytopes, {C}oxeter orbifolds
  and torus actions}, Duke Math. J. \textbf{62} (1991), no.~2, 417--451.

\bibitem[Dut]{Dutch20}
S.~Dutch, \emph{Enumeration of polyhedra},
  \url{https://stevedutch.net/symmetry/polynum0.htm}.

\bibitem[GL13]{Gi-LdM}
S.~Gitler and S.~{López de Medrano}, \emph{Intersections of quadrics,
  moment-angle manifolds and connected sums}, Geom. Topol. \textbf{17} (2013),
  no.~3, 1497--1534.

\bibitem[GVT{\etalchar{+}}18]{GVTFCLVGB:18}
P.~{Gómez-Gálvez}, P.~{Vicente-Munuera}, A.~Tagua, C.~Forja, A.M. Castro,
  M.~Letrán, A.~{Valencia-Expósito}, C.~Grima, and M.~{Bermúdez-Gallardo},
  \emph{Scutoids are a geometrical solution to three-dimensional packing of
  epithelia}, Nature Communications. 9 (1): 2960 \textbf{9} (2018), no.~3,
  2960.

\bibitem[Hae90]{H:90}
A.~Haefliger, \emph{Orbi-espaces}, Sur les groupes hyperboliques d'apr\`es
  {M}ikhael {G}romov ({B}ern, 1988), Progr. Math., vol.~83, Birkh\"{a}user
  Boston, Boston, MA, 1990, pp.~203--213.

\bibitem[HQ84]{HDu}
A.~Haefliger and {Quach N.D.}, \emph{Appendice: une pr\'{e}sentation du groupe
  fondamental d'une orbifold}, no. 116, 1984, Transversal structure of
  foliations (Toulouse, 1982), pp.~98--107.

\bibitem[J{\etalchar{+}}18]{binder}
Jupyter et~al., \emph{Binder 2.0 - {R}eproducible, interactive, sharable
  environments for science at scale}, \url{https://mybinder.org}, 2018, doi:
  10.25080/Majora-4af1f417-011.

\bibitem[{Ló}89]{LdM1989}
S.~{López de Medrano}, \emph{Topology of the intersection of quadrics in
  {$\mathbb{R}^n$}}, Algebraic topology ({A}rcata, {CA}, 1986), Lecture Notes
  in Math., vol. 1370, Springer, Berlin, 1989, pp.~280--292.

\bibitem[{Ló}14]{LdM2014}
\bysame, \emph{Singularities of homogeneous quadratic mappings}, Rev. R. Acad.
  Cienc. Exactas F\'{\i}s. Nat. Ser. A Mat. RACSAM \textbf{108} (2014), no.~1,
  95--112.

\bibitem[{Ló}17]{LdM-Pepe}
\bysame, \emph{Singular intersections of quadrics {I}}, Singularities in
  geometry, topology, foliations and dynamics, Trends Math.,
  Birkh\"{a}user/Springer, Cham, 2017, pp.~155--170.

\bibitem[{Ló}23]{LdM2023}
\bysame, \emph{Topology and geometry of intersections of ellipsoids in
  {$\mathbb{R}^n$}}, Grundlehren der mathematischen Wissenschaften, vol. 361,
  Springer, Cham, 2023.

\bibitem[Mic]{alloctahedra}
G.P. Michon, \emph{Counting polyhedra},
  \url{https://www.numericana.com/data/polyhedra.htm}.

\bibitem[Neu81]{Neumann1981}
W.D. Neumann, \emph{A calculus for plumbing applied to the topology of complex
  surface singularities and degenerating complex curves}, Trans. Amer. Math.
  Soc. \textbf{268} (1981), no.~2, 299--344.

\bibitem[Pan08]{Pa}
T.E. Panov, \emph{Cohomology of face rings, and torus actions}, Surveys in
  contemporary mathematics, London Math. Soc. Lecture Note Ser., vol. 347,
  Cambridge Univ. Press, Cambridge, 2008, pp.~165--201.

\bibitem[S{\etalchar{+}}24]{Sagemath}
W.A. Stein et~al., \emph{Sage {M}athematics {S}oftware ({V}ersion 10.5)}, The
  Sage Development Team, 2024, {\tt http://www.sagemath.org}.

\bibitem[VS93]{vinberg-sh:93}
\`E.B. Vinberg and O.V. Shvartsman, \emph{Discrete groups of motions of spaces
  of constant curvature}, Geometry, {II}, Encyclopaedia Math. Sci., vol.~29,
  Springer, Berlin, 1993, pp.~139--248.

\bibitem[Wal67]{Waldhausen1967}
F.~Waldhausen, \emph{Eine {K}lasse von {$3$}-dimensionalen
  {M}annigfaltigkeiten. {I}, {II}}, Invent. Math. \textbf{3} (1967), 308--333;
  ibid. {\bf 4} (1967), 87--117.

\bibitem[Wal80]{wall}
C.T.C. Wall, \emph{Stability, pencils and polytopes}, Bull. London Math. Soc.
  \textbf{12} (1980), no.~6, 401--421.

\end{thebibliography}

\newcommand{\etalchar}[1]{$^{#1}$}
 \providecommand\noopsort[1]{}
\providecommand{\bysame}{\leavevmode\hbox to3em{\hrulefill}\thinspace}
\providecommand{\MR}{\relax\ifhmode\unskip\space\fi MR }
\providecommand{\MRhref}[2]{%
  \href{http://www.ams.org/mathscinet-getitem?mr=#1}{#2}
}
\providecommand{\href}[2]{#2}

\end{document}